\documentclass[12pt]{article}
\usepackage{amssymb,amsmath,latexsym}

 \usepackage[dvips]{color}

\newtheorem{theorem}{Theorem}[section]
\newtheorem{lemma}[theorem]{Lemma}
\newtheorem{remark}[theorem]{Remark}
\newtheorem{example}{Example}[section]
\newtheorem{corollary}[theorem]{Corollary}
\newtheorem{definition}{Definition}[section]
\newtheorem{proposition}[theorem]{Proposition}
\newtheorem{assumption}{Assumption}[section]
\numberwithin{equation}{section}
\newcommand{\qed}{\hfill\rule{2mm}{3mm}\vspace{4mm}}

\def\qed{\hfill $\square$ \bigskip}

\def\beq{\begin{equation}}               %\beq=\begin{equation}
\def\eeq{\end{equation}}                 %\eeq=\end{equation}
\def\bea{\begin{eqnarray}}             %\beqn=\begin{eqnarray}
\def\eea{\end{eqnarray}}               %\eeqn=\end{eqnarray}
\def\be*{\begin{eqnarray*}}             %\beq=\begin{eqnarray*}
\def\ee*{\end{eqnarray*}}               %\eeq=\end{eqnarray*}
\def\ba{\begin{array}}                  %\ba=\begin{array}
\def\ea{\end{array}}                    %\ea=\end{array}

\def\;{\vspace{3mm} \\ }

\def\R{\mathbb{R}}

\def\P{\mathbb{P} }

\def\~{\widetilde}

\def\beqlb{\begin{eqnarray}} \def\eeqlb{\end{eqnarray}}
\def\beqnn{\begin{eqnarray*}} \def\eeqnn{\end{eqnarray*}}

\def\<{\langle}  \def\>{\rangle}

\def\bde{\begin{definition}}
\def\ede{\end{definition}}

\def\bth{\begin{theorem}}
\def\eth{\end{theorem}}
\def\ble{\begin{lemma}}
\def\ele{\end{lemma}}
\def\bcor{\begin{corollary}}
\def\ecor{\end{corollary}}

\def\pf{\noindent{\bf Proof.} }
\def\qed{{\hfill $\Box$ \bigskip}}

\begin{document}

\title
{\Large\bf Weak extinction versus global exponential growth of total mass for superdiffusions
}

\author{ \bf J\'{a}nos Engl\"{a}nder, \hspace{1mm}\hspace{1mm} Yan-Xia Ren
\footnote{The research of this author is supported by NSFC
(Grant No. 10971003, 11271030 and 11128101) and
Specialized Research Fund for the Doctoral Program of Higher Education.
\hspace{1mm} } \hspace{1mm}\hspace{1mm} and
Renming Song\thanks{Research supported in part by a
grant from the Simons Foundation (208236).}
\hspace{1mm} }
\date{}
\maketitle

\begin{abstract}
Consider a superdiffusion $X$ on $\mathbb R^d$ corresponding to
the semi-linear operator $\mathcal{A}(u)=Lu+\beta u-ku^2,$ where
$L$ is a second order elliptic operator, $\beta(\cdot)$ is in the Kato class,
and $k(\cdot)\ge 0$ is bounded on compact subsets of $\R^d$
and is positive on a set of positive Lebesgue measure.

The main purpose of this paper is to complement the results obtained in \cite{Englander:2004},
in the following sense.
Let $\lambda_\infty $ be  the $L^\infty$-growth bound of the
semigroup corresponding to the Schr\"odinger-type operator $L+\beta $.
If $\lambda_\infty \neq0$, then we prove that, in some sense,  the exponential growth/decay
rate of $\|X_t\|$, the total mass of $X_t$, is $\lambda_\infty $. We also describe the
limiting behavior of $\exp(-\lambda_\infty t)\|X_t\|$, as $t\to\infty$, in these cases. This should be
compared to the result in \cite{Englander:2004}, which says that the generalized principal eigenvalue $\lambda_2$ of
the operator gives the rate of {\it local} growth when it is positive, and implies local
extinction otherwise. It is easy to show that $\lambda_{\infty}\ge \lambda_2$, and we
discuss cases when $\lambda_{\infty}> \lambda_2$ and when $\lambda_{\infty}= \lambda_2$.

When $\lambda_\infty =0$, and under some conditions on $\beta$,  we give a sufficient
and necessary condition for the
superdiffusion  $X$ to exhibit weak extinction. We show that the branching intensity $k$
affects  weak extinction; this should be compared to the known result that $k$ does not
affect weak {\it local} extinction. (The latter  depends on the sign of $\lambda_2$ only,
and it turns out to be equivalent to local extinction.)
\end{abstract}
\tableofcontents

\medskip
\noindent {\bf AMS Subject Classifications (2000)}: Primary 60J80;
Secondary 60G57, 60J45

\medskip

\noindent{\bf Keywords and Phrases}: superdiffusion, superprocess,
measure-valued process, gauge theorem, Kato class, growth bound, principal eigenvalue,
$h$-transform, weak extinction, total mass.

\bigskip\bigskip

\baselineskip=6.0mm

\section{Introduction}
\setcounter{equation}{0}
\subsection{Model}
For any positive integer $i$ and $\eta\in (0, 1]$,
let $C^{i,\eta}(\R^d)$ denote the space of $i$ times
continuously differentiable functions with all their $i$-th order
derivatives belonging to $C^{\eta}(\R^d)$. (Here $C^{\eta}(\R^d)$
denotes the usual H\"older space.) For any
$x\in \R^d$, we will use $\{\xi_t, \Pi_{x}, t\ge
0\}$ to denote  the $L$-diffusion  with $\Pi_{x}(\xi_0=x)=1$, where
$$
L:=\frac{1}{2}\nabla\cdot
a\nabla+b\cdot\nabla\quad \mbox{ on }\R^d,
$$
and $a, b$ satisfy the following
\begin{description}
\item{(1)} the symmetric matrix $a=\{a_{i,j}\}$  satisfies
$$
A_1|v|^2\le\sum^d_{i,j=1}a_{i,j}(x)v_iv_j\le A_2|v|^2,\quad\mbox{ for
all }v\in\R^d\mbox{ and }x\in \R^d
$$
with some  $A_1, A_2>0$, and
 $a_{i,j}\in C^{1,\eta}, i,j=1,\cdots,d,$ for
some $\eta$ in $(0,1]$;
\item{(2)} the coefficients $b_i$, $i=1, \cdots, d$, are measurable functions satisfying
$$
\sum^d_{i=1}|b_i(x)|\le C(1+|x|), \qquad \mbox{ for
all } x\in \R^d
$$
with some $C>0$;
\item{(3)} there
exists a differentiable function $Q:\R^d\to \R$ such that $b=a\nabla
Q$.
\end{description}

\begin{remark}\label{conserv}\rm
Under (1)--(3) above, the diffusion process $\xi$ is {\it conservative} on $\R^d$. That is,
$$\Pi_x\left(\xi_t\in\R^d,\ \forall t>0\right)=1,$$
for all $x\in\R^d$; equivalently, the semigroup corresponding to $\xi$ leaves
the function $f\equiv 1$ invariant.
For a proof, see, for
instance, \cite[Theorem 10.2.2]{SV}.
It is well known that $\xi$ has a transition density $p(t, x, y)$ with respect to the Lebesgue
measure.
\end{remark}

Define
\begin{equation}\label{e:m}
m(x)=e^{2Q(x)}, \qquad x\in \R^d.
\end{equation}
Then $\xi$ is an $m$-symmetric Markov process, that is, the
semigroup of $\xi$ in $L^2(\R^d, m(x)\,\mathrm{d}x)$ is
symmetric in the sense that for any $t>0$ and $f, g\in L^2(\R^d, m(x)\,\mathrm{d}x)$,
$$
\int_{\R^d}f(x)\Pi_xg(\xi_t)m(x)\, \mathrm{d}x=\int_{\R^d}g(x)\Pi_xf(\xi_t) m(x)\, \mathrm{d}x.
$$
If $C^{\infty}_c(\R^d)$ denotes the space of infinitely differentiable functions with compact support,
then the
Dirichlet form $({\cal E}, D({\cal E}))$ of $\xi$ in $L^2(\R^d,
m(x)\,\mathrm{d}x)$ is the closure of the form given by
$$
{\cal E}(u, v)=\frac{1}{2}\int_{\R^d}(\nabla ua\nabla v)\exp(2Q)\mathrm{d}x,
\quad u,v\in C^{\infty}_c(\R^d).
$$
For any measurable space $(E, {\cal B})$, we denote by $M(E)$ the
set of all finite measures on ${\cal B}$, equipped with the weak topology.
We denote by ${\cal M}$ the Borel $\sigma$-field on $M(E)$, and so
${\cal M}$ is generated by all the functions $f_B(\mu)=\mu(B)$ with $B\in{\cal B}$.
The space of finite measures with compact support will be denoted by $M_c(E)$. The expression
$\langle f,\mu\rangle$ stands for the integral of $f$ with respect
to $\mu.$

With $\beta$ belonging to a certain Kato class (see Definition \ref{Kato}) and $k$
being locally  bounded
from above and nonnegative, we will define the fundamental quantity $\lambda_2$ in (\ref{e:l2def})
and show that $\lambda_2<\infty$.  We will use   $(\{X_t\}_{t\ge 0};
\P_{\mu},\,\mu\in M(\mathbb R^d))$
 to denote the {\it superprocess} (a measure-valued Markov process) with
$\P_{\mu}(X_0=\mu)=1$, corresponding to the semi-linear elliptic  operator
$\mathcal{A}(u):=Lu+\beta u-ku^2$ on $\R^d$.
For the precise definition, see Definition \ref{finiteMVP} below.
As we will see in Theorem \ref{existence},  the superprocess is  well defined.

\subsection{Motivation}
The main purpose of this paper is to complement the results obtained in
\cite{Englander:2004}. In particular, we study the growth/decay rate of
the total mass of $X$ and weak extinction\footnote{Some authors prefer to
say that $X$ `extinguishes.'} of $X$. Whereas in \cite{Englander:2004},
the local behavior of the mass has been shown to be
intimately related to the generalized principal eigenvalue  $\lambda_2$ corresponding to
the semigroup,
here we will show that the global behavior of the mass is linked to another important
quantity $\lambda_\infty$, the $L^{\infty}$-bound for the semigroup.

\subsection{Known results}
We first recall some definitions from Engl\"ander and Kyprianou
\cite{Englander:2004}.

\begin{definition}\label{extinction}\rm
Fix a nonzero $\mu\in{ M}(\R^d)$ with
compact support.

(i) We say that $X$ {\it exhibits local extinction} under $\P_{\mu}$ if for
every bounded Borel set $B\subset \R^d$, there exists a random time
$\tau_B$ such that
$$
\P_{\mu}(\tau_B<\infty)=1\quad \mbox{ and }\quad  \P_{\mu}(X_t(B)=0
\mbox{ for all } t\ge \tau_B)=1.$$

(ii) We say that $X$ {\it exhibits weak local extinction} under $\P_{\mu}$
if for every bounded Borel set $B\subset\R^d$,
$\P_{\mu}(\lim_{t\to\infty}X_t(B)=0)=1$.

(iii) We say that $X$ {\it exhibits extinction} under $\P_{\mu}$ if  there exists a stopping time
$\tau$ such that
$$
\P_{\mu}(\tau<\infty)=1\quad \mbox{ and }\quad  \P_{\mu}(X_t(\R^d)=0
\mbox{ for all } t\ge \tau)=1.$$

(iv) We say that $X$  {\it exhibits weak extinction} under $\P_{\mu}$ if
$\P_{\mu}(\lim_{t\to\infty}X_t(\R^d)=0)=1.$
\end{definition}

Let $\lambda_2$ be the growth bound of the semigroup in $L^2(\R^d, m)$
corresponding to the operator $L+\beta$ (see \eqref{e:l2def} and \eqref{prob-lambda}).
In \cite{Pinsky:1996}, Pinsky gave a criterion  for the local
extinction of $X$ under the assumption that $\beta$ is H\"older
continuous, namely, he proved that $X$ exhibits local
extinction if and only if $\lambda_2 \le 0$. In particular, local
extinction does not depend on the starting measure $\mu$ or the branching intensity $k$, but it
does depend on $L$ and $\beta$. (Note that, in regions where
$\beta>0$, $\beta$ can be considered as mass creation, whereas in
regions where $\beta<0$, $\beta$ can be considered as mass
annihilation.) Since local extinction depends on the sign of
$\lambda_2$, therefore, heuristically, it depends
on the competition between the outward speed of particles and the
mass creation.  The  main tools of \cite{Pinsky:1996} are  PDE
techniques.

In  \cite{Englander:2004}, Engl\"ander and Kyprianou presented
probabilistic (martingale and spine) arguments  for the fact that
$\lambda_2\le 0$ implies weak local extinction under $\P_{\mu}$ for any $\mu\in M(\R^d)$ with compact support, while $\lambda_2>0$
implies that, for any $\lambda <\lambda_2$ and any nonempty
relatively compact open set $B$,
$$
\P_{\mu}\left( \limsup_{t\to\infty}e^{-\lambda t}X_t(B)=\infty\right)>0
$$
holds for any nonzero initial measure $\mu$.

Putting things together, one concludes that in this case {\it local
extinction is in fact equivalent to weak local extinction} and there
is a  dichotomy in the sense that the process either exhibits local
extinction (when $\lambda_2\le 0$), or there is local exponential
growth with positive probability (when $\lambda_2>0$).

We will see that, on the other hand, extinction and weak extinction
are different in general. The intuition behind this is that the
total mass $\|X_t\|$ may stay positive but decay to
zero, {\it while   drifting out} (local extinction) and on its way
obeying changing  branching laws. (For a concrete example see
Example \ref{ex:new}.) This could not be achieved in a fixed compact
region with fixed branching coefficients.

Hence, weak extinction without extinction contrasts with the case without
spatial motion (continuous state branching processes), where such a
phenomenon requires a branching mechanism  which does not satisfy
the `Grey property' \cite{Grey:1974}.

In \cite{Englander:2004}
branching diffusions were studied besides superdiffusions, by using spine and
martingale methods. (Note that for
branching diffusions, weak (local) extinction and (local) extinction are
obviously the same, because the local/total mass is an integer.)
The main results concerned local extinction and local growth,
and  it was already noted that the growth rate of the  total mass may exceed
$\lambda_2 $ (see \cite[Remark 4]{Englander:2004}).

\subsection{Our main results}
It is important to point out that weak extinction, unlike local
extinction, depends on  the branching intensity $k$ as well
(see the $\lambda_\infty =0$ case below).
We will prove that the exponential growth rate of the
total mass is $\lambda_\infty $, defined by \eqref{def-Lambda}. More
precisely, there are three cases:
\begin{enumerate}
\item If mass creation is large enough so that $\lambda_\infty >0$,
then the total mass of $X$ tends to infinity exponentially with rate
$\lambda_\infty >0$, with positive probability. (Note that extinction always has a positive probability.)
\item  If annihilation is strong enough so that $\lambda_{\infty} <0$,
then the total mass of $X$ tends to zero exponentially with  rate
$\lambda_\infty <0$, a.s., even under survival.  (See Example \ref{ex:new} for
 a super-Brownian motion, where $\lambda_{\infty} <0$, but the process survives
with positive probability. Interestingly, as we will see in that example, having a
small $k$ term makes extinction avoidable,
while it cannot prevent weak extinction.)
\item If $\lambda_\infty =0$,
then weak extinction depends on  $k$.
\end{enumerate}
Concerning the third case, under some further conditions on $\beta$,
we will give a necessary and sufficient condition for $X$ to exhibit
weak extinction (see Remark \ref{r:1.4}).

Applying  our findings to the super-Brownian ($L=\frac{1}{2}\Delta$) case will
yield some interesting results; see  Section \ref{BM}.

In all the work mentioned above, $\beta$ is assumed to be H\"older
continuous. In this paper, we relax this condition by using results
of \cite{Chen:2002, Song:2002, Gesztesy:1991, Gesztesy:1995,
Zhao:1992} on Schr\"odinger operators.
The results of this paper are new even under the assumption that $\beta$
is H\"older-continuous. Furthermore, even under the H\"older-continuity assumption,
the arguments of this paper can not be
simplified by much.

Before we give the main results of this paper, let us introduce
some definitions and notation.

\begin{definition}[Kato class]\label{Kato}\rm  A measurable
function $q$ on $\R^d$ is said to be in the {\it Kato class} ${\bf K}(\xi)$  if
$$
\lim_{t\downarrow
0}\sup_{x\in\R^d}\Pi_x\left(\int^t_0|q(\xi_s)|\,\mathrm{d}s\right)=0.
$$
\end{definition}

It is easy to see that any bounded function is in the Kato class
${\bf K}(\xi)$. For any $q\in {\bf K}(\xi)$, denote
\begin{equation}
e_{q}(t):=\exp\left(\int^t_0q(\xi_u)\,\mathrm{d}u\right),
\end{equation}
and define
\begin{equation}e_{q}(\infty):=\exp\left(\int^{\infty}_0q(\xi_s)\,\mathrm{d}s\right),
\end{equation}
whenever the integral on the righthand side makes sense.
\begin{assumption}\label{Kato.assumption}
\rm In the remainder of this article, we will always assume that $\beta\in{\bf K}(\xi)$.
 \end{assumption}

One may define a semigroup $\{P^\beta_t\}_{t\ge 0}$ on
$L^p(\R^d, m)$, for any $p\in [1, \infty]$, by
$$
P^{\beta}_tf(x):=\Pi_x[e_{\beta}(t)f(\xi_t)].
$$
For any $p\in [1, \infty]$, $\|\cdot\|_p$ stands for the norm in $L^p(\R^d, m)$,
while $\|\cdot\|_{p, p}$ stands for
the operator norm from $L^p(\R^d, m)$ to $L^p(\R^d, m)$.
It follows from \cite[Theorem 3.10]{Chung:1995} that, for any $t>0$ and $p\in [1, \infty)$,
$\|P^{\beta}_t\|_{p, p}\le \|P^{\beta}_t\|_{\infty, \infty} \le e^{c_1t + c_2}$
for some constants $c_1, c_2$,
and that $\{P^\beta_t\}_{t\ge 0}$ is a strongly
continuous semigroup in $L^p(\R^d, m)$ for any $1\le p<\infty$. We
define
\begin{equation}\label{e:l2def}
\lambda_2(\beta):=\lim_{t\to\infty}\frac1{t}\log \|P^\beta_t\|_{2, 2}.
\end{equation}
\begin{remark}[Probabilistic representation]\rm
In fact, the following probabilistic characterization holds (see Appendix B):
\begin{equation}\label{prob-lambda}
\lambda_2(\beta) =\sup_{A\subset\subset \R^d}\lim_{t\to\infty}\frac{1}{t}
\log\sup_{x\in A}\Pi_x\left(
e_{\beta}(t); \tau_A>t\right).
\end{equation}
(Here $A\subset\subset \R^d$ means that $A$ is a bounded set in $\R^d$.)
In particular, $\lambda_2(0)$ is the `rate of escape from compacts' for the diffusion $\xi$.
In general, when $\beta$ is H\"older-continuous, $\lambda_2(\beta)$ coincides with the
so-called {\it generalized principal eigenvalue} of $L+\beta$ defined in \cite{Pinsky:1995}.
In our symmetric setting however,
for such a $\beta$, the situation is even simpler:  $\lambda_2(\beta)$ is the
{\it supremum of the $L^2$-spectrum} for the self-adjoint realization of the
symmetric operator $L+\beta$ on $\R^d$, obtained via the Friedrichs extension theorem.
(See \cite[Chap. 4]{Pinsky:1995}, especially Proposition 4.10.1 there, for more explanation).
\end{remark}

Now we recall the definition of an $(L,\beta, k)$-superprocess. For background
material on superprocesses, see
\cite{Dawson:1993,  Dynkin:1993, Dynkin:1994, Dynkin:2003, Li:2011}.

\begin{definition}[$(L,\beta, k)$-superprocess]\label{finiteMVP}\rm
An {\it $(L,\beta, k)$-superprocess}  is a measure-valued Markov process
$(\{X_t\}_{t\ge 0};\P_{\mu},\,\mu\in M(\mathbb R^d))$ such that
$\P_{\mu}(X_0=\mu)=1$, and  for any  bounded  Borel  $f\ge 0$ on $\R^d$,  one has
\begin{equation}\label{inhom-fund}
\P_{ \mu}\exp\langle -f,
X_{t}\rangle=\exp\langle -u(t, \cdot),\mu\rangle,
\end{equation} where $u$ is the minimal nonnegative solution to
\begin{equation}\label{inhom-int}
u(t, x)+\Pi_{x}\int^{t}_0k(\xi_s)(u(t-s, \xi_s))^2\mathrm{d}s-\Pi_{
x}\int^{t}_0\beta(\xi_s)u(t-s, \xi_s)\mathrm{d}s=\Pi_{x}f(\xi_{t}).
\end{equation}
\end{definition}
We will also say that $(\{X_t\}_{t\ge 0};\P_{\mu},\,\mu\in M(\mathbb R^d))$ is the
superprocess  `corresponding to the semi-linear elliptic  operator $\mathcal{A}(u):=Lu+\beta
u-ku^2$ on $\R^d$.'
\begin{theorem}[Existence]\label{existence}
 Suppose that $\beta\in{\bf K}(\xi)$ and $k\ge 0$ is locally  bounded.
Then the  $(L,\beta, k)$-superprocess exists.
\end{theorem}
\begin{remark}[Minimality and uniqueness]\rm Under our general condition on $k$,
we do not claim the uniqueness of the solution to the cumulant equation (\ref{inhom-int}).
In the Appendix, we will construct a minimal solution instead. If, however,
$k\in{\bf K}(\xi)$ holds as well, then the solution is unique, see Remark \ref{uniqueness}.
\end{remark}

Right after the construction of the superprocess, one of course would like to know
what regularity properties of the paths one can assume.
\begin{theorem}[Path regularity]\label{cadlag} Assume that $\beta\in{\bf K}(\xi)$
and is bounded from above, and $k\ge 0$ is locally  bounded. Then the superprocess
constructed in Theorem \ref{existence} has a version which possesses
$c\grave{a}dl\grave{a}g$ paths (that is, right continuous paths with
left limits, in the weak topology of measures).
\end{theorem}
The proofs of Theorems \ref{existence} and \ref{cadlag} are relegated to Appendix A.

Throughout this paper, the following assumption will be in force:
\begin{assumption}[Regularity assumption]\rm  The superprocess
$X$ has $c\grave{a}dl\grave{a}g$ paths.
\end{assumption}

\begin{remark}\rm
Note that, by Theorem \ref{cadlag}, the condition that $\beta$ is
bounded from above is a sufficient condition for the existence of a
regular version of $X$. What we need in the rest of this paper is the
existence of a regular version of $X$. With Assumption 1.2 in force,
we do not need to assume that $\beta$ is bounded from above
in the rest of this paper.
\end{remark}

Returning now to the analytic tools needed, another very important quantity besides
$\lambda_2$, is given in the following definition.
\begin{definition}[$L^{\infty}$-growth bound]\label{growth.bound}\rm
Define
\begin{equation}\label{def-Lambda}
\lambda_\infty(\beta) :=\lim_{t\to\infty}\frac{1}{t}\log\|P^\beta_t\|_{\infty, \infty}=
\lim_{t\to\infty}\frac{1}{t}\log\sup_{x\in
\R^d}\Pi_xe_{\beta}(t).
\end{equation}
We call $\lambda_{\infty}=\lambda_{\infty}(\beta)$ the {\it $L^{\infty}$-growth bound}.
\end{definition}

It follows from \eqref{prob-lambda} and \eqref{def-Lambda}
that $\lambda_\infty(\beta)\ge\lambda_2(\beta)$.
In fact,
$\lambda_\infty(\beta) =\lambda_2(\beta) $ and
$\lambda_\infty(\beta) >\lambda_2(\beta) $ are both possible. For
conditions  under which $\lambda_\infty(\beta) =\lambda_2(\beta) $,
we refer to  Chen \cite[Section 4]{Chen:2011} and the references
therein. We will give some examples  of $\lambda_\infty(\beta)
>\lambda_2(\beta) $ in Section \ref{sec-example}.

For simplicity, we will
write $\lambda_2(\beta)$ as $\lambda_2$, and $\lambda_\infty(\beta)$
as $\lambda_\infty$ when the potential $\beta$ is fixed.

The following notion is of fundamental importance.

\begin{definition}[Gauge function]\label{def.of.Gauge.function}
\rm For any $\beta\in {\bf K}(\xi)$, we define
\begin{equation}\label{definition-g}
g_\beta(x)=\Pi_x(e_{\beta}(\infty)),
\quad x\in\R^d,
\end{equation}
when the right hand side is well defined. The function $g_\beta$, called
the {\it gauge function}, is very useful in studying the potential theory of
the Schr\"odinger-type operator $L+\beta$.
\end{definition}

We are now ready to state the main results of this paper, the first of which treats the
`over-scaling' and  `under-scaling' of the total mass  $\|X_t\|:=\langle 1,
X_t\rangle$.
 \begin{theorem}[Over- and under-scaling]\label{th1} Let  $\mu\in {M}(\R^d)$
be nonzero.
\begin{description}
 \item{(1)}
 For any $\lambda> \lambda_\infty$,
\begin{equation}
\P_{\mu}\left(\lim_{t\to\infty}e^{-\lambda t}\|X_t\|=0\right) =1.
\end{equation}
In particular, if $\lambda_{\infty}<0$, then $X$ suffers weak extinction.
\item{(2)} Assume that $k$ is bounded. If $\lambda_\infty>0$ and
\begin{equation}\label{uniform-bd}
\liminf_{t\to\infty}
\frac{\Pi_xe_{\beta}(t)}{\sup_{y\in \R^d}\Pi_ye_{\beta}(t)}>0
\quad \mbox{ for all } x\in \R^d
\end{equation}
holds, then for any $\lambda< \lambda_\infty$,
\begin{equation}\label{limsup=infty2}
\P_{\mu}\left(\limsup_{t\to\infty}e^{-\lambda t}\|X_t\|=\infty\right) >0.
\end{equation}
\end{description}
\end{theorem}

Condition \eqref{uniform-bd} is  rather restrictive. It is certainly satisfied when $\beta$ is
a constant. Using Lemma \ref{limsup-infty} below, one can come up with many examples of
non-constant functions
satisfying this condition.

The next two theorems give some insight as to what happens when the scaling
of the total mass is exactly at $\lambda_{\infty}$.
Obviously, the conditions in the next two theorems are not optimal. We
plan to establish more general
versions of these two theorems in an upcoming paper.

\begin{theorem}[Scaling  at $\lambda_{\infty}$]\label{th2}
Let  $\mu\in {M}(\R^d)$ be nonzero.
\begin{description}
\item{(1)} Assume that $\lambda_\infty>0$ and  that \eqref{uniform-bd} holds.
If
\begin{equation}\label{g(t)-to-infty}
\lim_{t\to\infty}\Pi_xe_{\beta-\lambda_\infty}(t)=\infty
\quad \mbox{ for all } x\in\R^d,
\end{equation}
then
\begin{equation}\label{limsup=infty1}
\P_{\mu}\left(\limsup_{t\to\infty}e^{-\lambda_\infty t}\|X_t\|=\infty\right) >0.
\end{equation}
\item{(2)} If $g_{\beta-\lambda_\infty}(x)\equiv 0$ in $\R^d$ and
\begin{equation}\label{upper-domi}
\sup_{x\in\R^d}\Pi_{x}\left(\sup_{t\ge 0}e_{\beta-\lambda_\infty }(t)\right)<\infty,
\end{equation}
then
\begin{equation}\label{liminf=0}
\P_{\mu}\left(\liminf_{t\to\infty}e^{-\lambda_\infty t}\|X_t\|
=0\right)=1.
\end{equation}
If, in addition, $\beta\le 0$ on $\mathbb R^d$, then
the superprocess suffers weak extinction.
\end{description}
\end{theorem}
\begin{remark}\rm Assuming $g_{\beta-\lambda_{\infty}}\equiv\infty$ would automatically imply \eqref{g(t)-to-infty}.
\end{remark}

Unlike in the previous two results, the next two involve the coefficient $k$ as well.

The result below relates scaling and positive solutions (in the sense of distributions) of
 $(L+\beta-\lambda_\infty)h=0$. Recall that a function $h$ is a solution to
 $(L+\beta)h=0$ in the sense of distributions if the generalized derivative
$\nabla h$ is locally $L^2$-integrable with respect to $m(x)\mathrm{d}x$ and for any
$\varphi\in C_c^\infty(\mathbb R^d)$,
$$
\frac{1}{2}\int_{\R^d}(\nabla ha\nabla \varphi)\exp(2Q)\mathrm{d}x-\int_{\R^d}
h(x)\varphi(x)\beta(x)\mathrm{d}x=0.
$$

\begin{theorem}\label{th3}
Assume that there is a bounded solution
$h>0$ of $(L+\beta-\lambda_\infty)h=0$ in $\R^d$ in the sense of distributions.
If there exists an $x_0\in\R^d$ such that
\begin{equation}\label{<infty2}\Pi_{x_0}\int^{\infty}_0e_{\beta-2\lambda_\infty
}(s)k(\xi_s)\mathrm{d}s<\infty,
\end{equation}
then
$\lim_{t\to\infty}e^{-\lambda_\infty t}\langle h, X_t\rangle$
exists $\P_\mu$-a.s. and in $L^2(\P_{\mu})$,
and $\P_\mu(\|X_t\|>0,\ \forall t>0)>0$ for
all nonzero measures $\mu\in M_c(\R^d)$.
If, in addition, $h$ satisfies that
\begin{equation}\label{bounded-h}
\inf_{x\in\R^d}h(x)>0,
\end{equation}
then the scaling at $\lambda_{\infty}$ is the correct one in the sense that for every
nonzero $\mu\in{M}_c(\R^d)$,
\begin{equation}\label{limsup=infty}
\P_{\mu}\left(\limsup_{t\to \infty}e^{-\lambda_{\infty} t}\|X_t\|<\infty\right)=1
\end{equation}
and
\begin{equation}\label{limsup=0}
\P_{\mu}\left(\liminf_{t\to \infty}e^{-\lambda_{\infty}t}\|X_t\|>0\right)>0.
\end{equation}
\end{theorem}

\begin{remark}[On condition (\ref{<infty2})]\rm
Assume that the coefficients are smooth and $h>0$ is a strong solution in Theorem \ref{th3}.
From the fact that the operator $(L+\beta-\lambda_{\infty})^h$  defined by
$$
(L+\beta-\lambda_{\infty})^hu(x)=\frac1{h(x)}(L+\beta-\lambda_{\infty})(uh)(x)
$$
has no potential (zeroth order)
part, it follows that $$\Pi_{x_{0}}e_{\beta}(s) h(\xi_s)\le e^{\lambda_{\infty}s}h(x_0).$$
Thus, if $k\le Ch$, then
\begin{equation*}\Pi_{x_0}\int^{\infty}_0e_{\beta-2\lambda_\infty
}(s)k(\xi_s)\mathrm{d}s\le C_1\int_0^{\infty} e^{-\lambda_{\infty}s}\, \mathrm{d}s.
\end{equation*}
Consequently if $\lambda_{\infty}>0$ and $k/h$ is bounded from above (in particular,
if $k\in C_c(\mathbb R^d)$), then condition (\ref{<infty2}) is automatically satisfied.

Similarly, if $f>0$ solves $(L-\lambda_2(0))f=0$ (such a positive harmonic
function always exists if $L$ has smooth coefficients), then
$$
\Pi_{x_{0}}f(\xi_s)\le e^{\lambda_{2}(0)s}f(x_0).
$$
Suppose now that $\beta\equiv B$, where $B$ is an arbitrary constant.
Since $\xi$ is conservative, $\lambda_{\infty}=B$. So,
if $k\le Cf$ (in particular, if $k\in C_c(\mathbb R^d)$), then
\begin{equation}\label{also.very.good}
\Pi_{x_0}\int^{\infty}_0 e_{\beta-2\lambda_\infty }(s)k(\xi_s)\mathrm{d}s\le
C_1\int_0^{\infty} e^{(-B+\lambda_2(0))s}\, \mathrm{d}s.
\end{equation}
If $B>0$, then the integral on the righthand side of (\ref{also.very.good})
is always finite (since $\lambda_2(0)\le 0$), and
so condition (\ref{<infty2}) is automatically satisfied.

If $B\le 0$, it is still satisfied as long as
$|B|<|\lambda_2(0)|$,
that is, when the motion is sufficiently transient. To give a concrete example,
consider an `outward' Ornstein-Uhlenbeck process, with parameter $\gamma>0$,
corresponding to the operator
$$L=\frac{1}{2}\Delta +\gamma x\cdot \nabla\quad \text{on}\  \R^d.$$ Since
$\lambda_2=-\gamma d$, what we need  is $0<B+\gamma d$.
\end{remark}

We now present  a partial converse to Theorem \ref{th3}.
To state this result, we need to introduce another function class.
We note that the Kato class $\mathbf{K}$ introduced in Definition \ref{Kato} was
defined by a local condition, while
the class $\mathbf{K}_{\infty}$ introduced below is defined by a {\it global} condition.

\begin{definition}[The class $\mathbf{K}_{\infty}(\xi)$]\label{Kinfty}\rm
Assume that $\xi$ is transient. A function $q\in {\bf K}(\xi)$ is {\it said to be in
the class ${\bf K}_\infty(\xi)$} if for any $\epsilon>0$
there exist a compact set $K$ and a constant $\delta>0$ such that
for any subset $A$ of $K$ with $m(A)<\delta$,
\begin{equation}
\sup_{x\in \R^d}\int_{(\R^d\setminus
K)\cup A}\widetilde{G}(x,y)|q(y)|m(y)\mathrm{d}y<\epsilon,
\end{equation}
where $m$ is the function defined in \eqref{e:m} and
$\widetilde{G}(x,y)$ is the Green function corresponding to $\xi$ with respect to
$m(x)\mathrm{d}x$ in $\R^d$.
\end{definition}

The class ${\bf K}_\infty(\xi)$ was first introduced in \cite{Song:2002,
Chen:2002}. When $\xi$ is transient and  $\beta\in {\bf K}_\infty(\xi)$,
we have $\lambda_\infty\ge 0$. In fact, it follows from
\cite[Proposition 2.1]{Song:2002} that
$\Pi_x\left(\int^{\infty}_0|\beta|(\xi_s)ds\right)$ is  bounded in
$\R^d$.  Let $M$ be the upper bound. By Jensen's inequality, we have
\begin{equation}\label{e:jensen}
\Pi_xe_\beta(t)\ge\exp\left(-\Pi_x\int^{\infty}_0|\beta|(\xi_s)\mathrm{d}s\right)\ge e^{-M},
\end{equation}
which implies that
$$
\frac{1}{t}\log\sup_{x\in\R^d}\Pi_xe_\beta(t)\ge -M/t.
$$
Thus by definition,
$$
\lambda_\infty=\lim_{t\to\infty}\sup_{x\in\R^d}\frac{1}{t}\log\Pi_xe_\beta(t)\ge 0.
$$
Note that \eqref{e:jensen} implies that $g_\beta\ge e^{-M}$.
It follows from the gauge theorem (see \cite[Theorem 2.2]{Song:2002}
or \cite[Theorem 2.6]{Chen:2002}) that, if $\xi$ is transient and
$\beta\in {\bf K}_\infty(\xi)$, then $g_\beta$ is either bounded or identically
infinite. It follows from \cite[Corollary 2.9]{Chen:2002} that the
boundedness of $g_\beta$ implies that $\sup_{x\in
\R^d}\Pi_x(\sup_{t\ge 0}e_\beta(t))<\infty$ for every $x\in \R^d$,
and hence $\lambda_\infty(\beta)=0$.

Recall that a function $f$ on $\R^d$ is said to be {\it radial} if there exists some function
$\widetilde{f}$ on $[0, \infty)$ such that $f(x)=\widetilde{f}(|x|)$ for all $x\in \R^d$.

\begin{theorem}[Weak extinction in the radial case]\label{main2}
Assume that $k$ and $\beta$ are radial functions,
and $L$ is radial (i.e., $a_{i,j}$, $i,j=1, 2,\cdots, d$, and $Q$
are radial functions). Assume that $\xi$ is transient, $\beta\in
{\bf K}_\infty(\xi)$, and that $g_{\beta}(x)$ is not
identically infinite (which implies that
$g_{\beta}$ is bounded and hence $\lambda_\infty=0$). If
\begin{equation}\label{bL2norm=infty}
\Pi_{x}\left[
\int^{\infty}_0e_{\beta}(s)k(\xi_s)\mathrm{d}s\right]=\infty\quad\mbox{ for all
}x\in \R^d,
\end{equation}
then for every $\mu\in M(\R^d)$,
\begin{equation}\label{limt=0'}
\P_{\mu}\left(\lim_{t\to\infty}\|X_t\|=0\right)= 1.\end{equation}
\end{theorem}

\begin{remark}\label{r:1.4} \rm In particular, if $\xi$ is transient, $\beta\in
{\bf K}_\infty(\xi)$ and $g_{\beta}$ is not identically infinite, then
$g_\beta$ is a solution of $(L+\beta) u=0$ in the distribution sense, and is bounded between two
positive numbers (see the paragraphs after \eqref{e:jensen}). In this case,
Theorem \ref{th3} and Theorem
\ref{main2} imply that condition \eqref{bL2norm=infty} is a
necessary and sufficient condition for $X$ to exhibit weak extinction.
\end{remark}

In Section \ref{sec-example} we will give some examples for which
the conditions of our theorems are satisfied.
The assumption that $k, \beta, L$ are radial in Theorem \ref{main2} is rather restrictive.
We expect that an appropriate version of Theorem \ref{main2} will be
valid in the non-radial case too; we
plan to address this problem in an upcoming project.

\subsection{Outline}
The rest of the paper is organized as follows. In the next section we illustrate
our results with examples.
In the two sections following the examples, we provide the proofs. Those proofs
utilize some known results from Gauge Theory, as well as probabilistic
techniques. We presume that
the probabilistic audience likely to read this article would prefer to see the
(largely probabilistic) proofs of the results
without first being halted by a lengthy read about the technicalities of Gauge Theory.
Therefore, in order to make the material
presented easier to digest, we  relegate those technical lemmas into Appendix B.
In the same vein, to make the paper less overwhelmed by technical details at the
beginning, we defer the proof of path regularity to Appendix A.
The reader may consider, of course, to read the appendices right after reading the main results.

\section{Examples}\label{sec-example}

\subsection{Some super-diffusions  with $\lambda_\infty>\lambda_2$}
We start with an example in one dimension and with constant mass creation.
\begin{example}\rm
Consider the elliptic operator  $$L=\frac{1}{2}
\frac{\mathrm{d}^2}{\mathrm{d}x^2} -b_0\frac{\mathrm{d}}{\mathrm{d}x}$$ on $\R$,
where $b_0>0$ is a constant. Then the diffusion corresponding to $L$ is conservative and transient.
It is easy to see that the corresponding generalized principal
eigenvalue is $\lambda_2(0)=-b^2_0/2$. Let the potential  $\beta$  be a
nonnegative constant. We have
$\lambda_2(\beta)=\beta -b^2_0/2$
and $\lambda_\infty(\beta)=\beta$. The Green function of $\xi$  is
$G(x,y)=\frac{2\pi}{b_0}\exp\left(-2b_0(x-y)^+\right).$  Note that
$L-\beta+\lambda_\infty(\beta)=L$.

For the large time behavior of $X$ the following hold.

(i) According to   \cite[Theorem 7 and Example 1]{Pinsky:1996}, $X$ exhibits local
extinction if and only if $\beta \in[0,
b^2_0/2]$. Furthermore, when $\beta\in (b^2_0/2, \infty)$, $X$ does not exhibit
local extinction, and the exponential expected growth rate of the local mass
is $(\beta-b^2_0/2)$. More precisely, for any continuous function
$g$ on $\R$ with compact support and any  nonzero $\mu\in M_c(\R)$, one has
$$
\lim_{t\to\infty}e^{\rho t}\P_{\mu}\langle g,
X_t\rangle=\left\{\begin{array}{rl}0,\quad
&\varrho\le -(\beta-b^2_0/2),\\
+\infty,\quad &\varrho>-(\beta-b^2_0/2).\end{array}\right.
$$
In fact, by  \cite{Englander:2004}, the local mass grows exponentially with
positive probability, that is, not just in expectation.

 (ii) If $\beta>0$,
since $\Pi_xe_{\beta}(t)=e^{\beta t}$ for all $ x\in{\R}$
and $t\ge 0$, \eqref{uniform-bd} is satisfied. Thus by Theorem
\ref{th1}, we have that,  for any $\lambda> \beta$,
$$
\P_{\mu}\left(\liminf_{t\to\infty}e^{-\lambda t}\|X_t\|=0\right)
=1,
$$
and that if  $k$ is bounded, then, for any $\lambda<
\beta$,
$$
\P_{\mu}\left(\limsup_{t\to\infty}e^{-\lambda t}\|X_t\|=\infty\right) >0.
$$

(iii) Since  $u\equiv 1$ solves $Lu=0$, by Theorem
\ref{th3},  if  there exists an $x_0\in\R$ such that
\begin{equation}
\Pi_{x_0}\int^{\infty}_0e^{-\beta s}k(\xi_s)\, \mathrm{d}s<\infty,
\end{equation}
then for any nonzero $\mu\in{M_c}(\R^d)$, the limit
$\lim_{t\to \infty}\exp(-\beta t)\|X_t\|$ exists
$\P_\mu$-a.s. and in $L^2(\P_\mu)$, and
$$
0<\P_{\mu}\left(\left[\lim_{t\to \infty}\exp(-\beta t)\|X_t\|\right]^2\right)<\infty.
$$
Hence,
$$
\P_{\mu}\left(\lim_{t\to \infty}\exp(-\beta t)\|X_t\|=0\right)<1,
$$
and
$$
\P_{\mu}\left(\lim_{t\to \infty}\exp(-\beta t)\|X_t\|=\infty\right)=0.
$$

(iv) Since $L$ is radial, by Theorem \ref{main2} we have that in the case of critical branching
($\beta=0$),  if
\begin{equation}\label{iif-cond}
\int^{x}_{-\infty}\exp\left(-b_0(x-y)\right)k(y)\mathrm{d}y+\int^{\infty}_xk(y)
\mathrm{d}y=\infty,\quad
x\in\R,
\end{equation}
then
$$
\P_{\mu}\left(\lim_{t\to\infty}\|X_t\|=0\right)= 1.$$
In summary,
\begin{enumerate}
\item [(a)] If $\beta>0$,  the exponential growth rate of
the total mass is $\beta$.
\item [(b)] If $\beta=0$, weak extinction depends on
the branching rate function $k$: the superprocess exhibits weak
extinction if and only if \eqref{iif-cond} holds.
\end{enumerate}

\end{example}
In the next example the motion component is a multidimensional  `outward Ornstein-Uhlenbeck' process.
\begin{example}\rm
Consider the elliptic operator   $$L=\frac{1}{2}\Delta +\gamma x\cdot
\nabla\quad \text{on}\  \R^d,$$
where $d\ge 1$ and $\gamma>0$. Then the diffusion corresponding to $L$ is
conservative and transient, and $\lambda_2(0)=-\gamma d$.
Let the potential  $\beta$  be a positive constant. Then
$\lambda_2(\beta) =\beta-\gamma d$, and  $\lambda_\infty(\beta) =\beta$.

(i) $X$ exhibits local extinction if and only if $\beta \in[0,
\gamma d]$. If $\beta\in (\gamma d, \infty)$, then $X$ does not exhibit local
extinction, and the exponential growth rate of the local mass is
$\beta-\gamma d$. More precisely, for any continuous function $g$ on
$\R^d$ with compact support,
$$
\lim_{t\to\infty}e^{(\beta-\gamma d)t}\langle g, X_t\rangle=N_{\mu}\int_{\R^d}
g(x)\exp(-\gamma|x|^2/2)\mathrm{d}x,
\quad \mbox{ in }\P_{\mu}\mbox{--probability}
$$
for some random variable $N_{\mu}$
with mean $\int_{\R^d}\exp(-\gamma|x|^2/2)\mu(\mathrm{d}x)$
whenever there exists a $K>0$ such that
$$
k(x)\le K\exp(\gamma|x|^2/2),\quad \mbox{ for all } x\in \R^d,
$$
and the starting measure $\mu=X_0$ satisfies
$$
\int_{\R^d}\exp(-\gamma|x|^2/2)\mu(\mathrm{d}x)<\infty.
$$
See \cite[Theorem 1]{Englander:2006} and  \cite[Example 23]{Englander:2002}.

(ii) By Theorem \ref{th1}, we have that, for any $\lambda> \beta$,
$$
\P_{\mu}\left(\liminf_{t\to\infty}e^{-\lambda t}\|X_t\|=0\right)
=1,
$$
and that if  $k$ is bounded in $\R^d$, then, for any $\lambda<
\beta$,
$$
\P_{\mu}\left(\limsup_{t\to\infty}e^{-\lambda t}\|X_t\|=\infty\right) >0.
$$

(iii) Obviously, $u\equiv 1$ is a bounded solution to $Lu=0$, and
by  Theorem \ref{th3} and its proof, we have that if the branching rate
$k$ satisfies
$$
\Pi_x\int^{\infty}_0e^{-\beta s}k(\xi_s)\, \mathrm{d}s<\infty,\quad x\in\R^d,
$$
then for  any nonzero $\mu\in{M_c}(\R^d)$, there exists
$\lim_{t\to \infty}\exp(-\beta t)\|X_t\|$
$\P_\mu$-a.s., and
$$
\P_{\mu}\left[\lim_{t\to \infty}\exp(-\beta t)\|X_t\|\right]^2\in(0,\infty).
$$
Hence,
$$
\P_{\mu}\left(\lim_{t\to \infty}\exp(-\beta t)\|X_t\|=0\right)<1,
$$
and
$$
\P_{\mu}\left(\lim_{t\to \infty}\exp(-\beta t)\|X_t\|=\infty\right)=0.
$$
\end{example}

\subsection{Extinction and weak extinction}

Next is an example illustrating the difference between extinction and weak extinction.
The superprocess $X$ below exhibits local extinction and also weak extinction,
nevertheless it survives with positive probability.

\begin{example}[Weak and also local extinction, but survival]\label{ex:new}\rm
Let $B,\epsilon>$ and consider the super-Brownian motion in $\R$
with $\beta(x)\equiv -B$ and $k(x)=\exp\left[\mp\sqrt{2(B+\epsilon)}x\right],$ that is, let $X$
correspond to the semi-linear elliptic operator $\mathcal{A}$, where
$$
\mathcal{A}(u):=\frac{1}{2}\frac{\mathrm{d}^2 u}{\mathrm{d}x^2}-Bu-\exp\left[\mp\sqrt{2(B+\epsilon)}x\right]u^2.
$$
By Theorem \ref{th1}, $X$ suffers weak extinction: for any $\delta>0$,
$$
\lim_{t\to 0}e^{(B-\delta)t}\|X_t\|=0.
$$
Also, clearly, $\lambda_2=-B$, yielding that $X$ also exhibits local extinction.

Now we are going to show that, despite the above, the process $X$ survives
with positive probability, that is
$$\mathbb P_{\mu}(\|X_t\|>0,\ \forall\ t>0)>0,$$
for any nonzero $\mu\in M(\R^d)$. In order to do this, we will use the definition and
basic properties of
 $h$-transforms and weighted superprocesses. These can be found in Section 2 of \cite{Pinsky}.

The function
$h(x):=e^{\pm \sqrt{2(B+\epsilon)}x}$ transforms the operator
$\mathcal{A}$  into $\mathcal{A}^h$, where
$$
\mathcal{A}^h(u):=\frac{1}{h}\mathcal{A}(hu)=\frac{1}{2}
\frac{\mathrm{d}^2 u}{\mathrm{d}x^2}\pm \sqrt{2(B+\epsilon)}
\frac{\mathrm{d}u}{\mathrm{d}x}+\epsilon u-u^2.
$$
(Note that $h''/2-(B+\epsilon)h=0$). The superprocess $X^h$ corresponding to $\mathcal{A}^h$
is in fact the same as the original process $X$, weighted by the function $h$, and consequently,
survival (with positive probability) is invariant under $h$-transforms.  But  $X^h$ has a
conservative motion component and  constant branching mechanism, which is supercritical,
and therefore  $X^h$  survives with positive probability; the same is then true for $X$. 
%$\hfill\diamond$
\end{example}

\subsection{The super-Brownian motion case}\label{BM}

In this subsection we focus on the special case when the underlying motion process is a
Brownian motion, that is, when $L=\Delta/2$; in the remainder of this section we will
always assume that this is the case.
In this case $\beta\in {\bf K}(\xi)$ if and only if
$$
\lim_{r\to 0}\sup_{x\in \R^d}\int_{|y-x|<r}u(x-y)|\beta(y)|\, \mathrm{d}y=0,
$$
where $u$ is the function defined by
\begin{equation}\label{e:gfn4bm}
u(x):=\begin{cases}|x|^{2-d}, ~&d\ge3\\
\log |x|^{-1}, ~&d=2\\
|x|, ~ &d=1.
\end{cases}
\end{equation}
When $d\ge 3$, ${\bf K}_\infty(\xi)$
coincides with the class ${\bf K}^\infty_d$ defined in \cite{Zhao:1992}.
We recall the definition of the class ${\bf K}^\infty_d$ defined in
\cite{Gesztesy:1991, Gesztesy:1995} in the case $d\le 2$.

\begin{definition}[The classes  ${\bf K}^\infty_1(\xi)$ and
${\bf K}^\infty_2(\xi)$]\rm Let $L=\Delta/2$.
\begin{description}
\item{(1)} If  $d=1$, a function $q\in {\bf K}(\xi)$ {\it is said to
be in the class ${\bf K}^\infty_1(\xi)$} if
$$
\int_{|y|\ge 1}|yq(y)|\,\mathrm{d}y<\infty.
$$
\item{(2)} If  $d=2$, a function $q\in {\mathbf K}(\xi)$ is {\it said to
be in the class ${\bf K}^\infty_2(\xi)$} if
$$
\int_{|y|\ge 1}\ln(|y|)|q(y)|\,\mathrm{d}y<\infty.
$$
\end{description}
\end{definition}

\subsubsection{The $d\ge 3$ case}

%We first recall the following definition from \cite{Simon:1981}.
Recall we have proved, in the paragraph below Definition 1.6,
that for any $\beta\in {\bf K}_\infty(\xi)$ we have $\lambda_\infty(\beta)\ge 0$.
The following definition is from \cite{Simon:1981}.

\begin{definition}[Criticality in terms of $\lambda_{\infty}$]\label{d:crit}\rm Let
%$L=\Delta/2$ and $\beta\in {\bf K}(\xi)$. Then $\beta$ is said to be
$L=\Delta/2$ and $\beta\in {\bf K}_\infty(\xi)$. Then $\beta$ is said to be
\begin{description}
\item{(a)}
{\it supercritical} iff $\lambda_\infty(\beta)>0$,
\item{(b)}
{\it critical} iff $\lambda_\infty(\beta)= 0$ and for any
nontrivial nonnegative continuous function $q$ of compact support,
$\lambda_\infty(\beta+q)>0$.
\item{(c)}
{\it subcritical} iff it neither supercritical nor critical.
\end{description}
\end{definition}
{\bf Note:} The reader should not confuse the above properties of the function $\beta$
with the (local) criticality (or sub- or supercriticality) of the branching, which
simply refer to the sign of $\beta$ (in certain regions).

The following result relates the above definition to the solutions of
\begin{equation}\label{harmonicity}
(L+\beta)u=0,
\end{equation}
and is due to \cite{Zhao:1992}.
\begin{lemma}\label{Zhao92} Let  $L=\Delta/2$, $\beta\in {\bf K}_\infty(\xi)$ and
$d\ge 3$. Then the following conditions are equivalent:
\begin{description}
\item{(a)} $\beta$ is subcritical.
\item{(b)} $g_\beta(x)\equiv\Pi_xe_\beta({\infty})$ is bounded in $\R^d.$
\item{(c)} There exists a solution $u$ to (\ref{harmonicity}) with
$\inf_{x\in\R^d}u(x)>0$.
\item{(d)} There exists a solution $u$ to (\ref{harmonicity}) with
$0<\inf_{x\in\R^d}u(x)\le\sup_{x\in\R^d}u(x)<\infty$.
\end{description}
Moreover, if $\beta$ is  subcritical, then  (\ref{harmonicity})
has a unique (up to constant multiples) positive bounded
solution and the solution must be of the form $cg_{\beta}(x)$ for some $c>0$.
\end{lemma}

However, if $\beta$ is critical, then there is no positive solution to \eqref{harmonicity} which is bounded away from zero. Pinchover
\cite{Pinchover:1995} proved the following result (see
\cite[Lemma 2.7]{Pinchover:1995}).

\begin{lemma}\label{exists-pb-solution} Let $L=\Delta/2$, $\beta\in {\bf K}_\infty(\xi)$ and
$d\ge 3$. If $\beta$ is critical, then there is an
$h>0$ satisfying  (\ref{harmonicity}) on $\R^d$ and such that
\begin{equation}\label{h-infty}
h\sim c_d|x|^{2-d}, \quad  \text{as}\ |x|\to\infty,
\end{equation}
 where $c_d$ is a positive
constant depending only on $d$.
\end{lemma}

It is easy to check that, for any $p>d/2$, $\beta\in L^1(\R^d)\cap L^p(\R^d)$ implies
that $\beta\in {\bf K}_\infty(\xi)$. In this special case,
the following result shows that $h$ can be obtained as a large time
asymptotic limit of the Schr\"odinger semigroup (see \cite[Theorem 3.1]{Simon:1981}

\begin{lemma}\label{Simon}  Let  $L=\Delta/2$, $\beta\in L^1(\R^d)\cap
L^p(\R^d)$ and $d\ge 3$. If $\beta$ is critical, then
\begin{equation}\label{sup<infty}
\lim_{t\to\infty}f(t)^{-1}\sup_{x\in\R^d}\Pi_x[e_{\beta}(t)]=C,
\end{equation}
and
\begin{equation}\label{pointlimit}
\lim_{t\to\infty}f(t)^{-1}\Pi_x[e_{\beta}(t)]=h(x),\quad\forall x\in\R^d,
\end{equation}
where  $C$ is a positive constant, $h>0$  is  bounded and solves
(\ref{harmonicity}) (general theory implies, in the critical case, the existence of such a solution) and
\begin{equation}\label{def-f(t)}
f(t)=\left\{\begin{array}{lr}t,\quad &d\ge
5,\\t/(\ln t), \quad &d=4,\\t^{1/2}\quad &
d=3.\end{array}\right.
\end{equation}
\end{lemma}

\begin{lemma} \label{limsup-infty}
Let $L=\Delta/2$ and  $d\ge 3$. If $\lambda_\infty(\beta)>0$ and
$\beta-\lambda_\infty\in L^1(\R^d)\cap L^p(\R^d)$, then conditions
\eqref{uniform-bd}  and \eqref{g(t)-to-infty} are satisfied.
\end{lemma}

\pf Note that
$$
g_{\beta}(t)=\sup_{x\in\R^d}\Pi_{x}e_{\beta}(t)=
e^{\lambda_\infty t}\sup_{x\in\R^d}\Pi_{x}e_{\beta-\lambda_\infty}(t).
$$
By Lemma \ref{Simon} we have
$$
 g_{\beta}(t)\sim
Ce^{\lambda_\infty t}f(t),\quad\mbox{ as } t\to\infty
$$
with $f(t)$ defined by \eqref{def-f(t)}, and
$$
\lim_{t\to\infty}g^{-1}_{\beta}(t)\Pi_xe_{\beta}(t)=\frac{1}{C}
\lim_{t\to\infty}f^{-1}(t)\Pi_xe_{\beta-\lambda_\infty}(t)>0,
$$
which means  that  conditions \eqref{uniform-bd} and
\eqref{g(t)-to-infty} are satisfied.\qed

This subsection shows that there are many examples of $\beta$ satisfying the conditions of
Theorems \ref{th1}--\ref{th2}(1).

\subsubsection{The $d\le 2$ case}

The purpose of this subsection is to show that the assumptions of
Theorem \ref{th2}(2) are satisfied for some super-Brownian motions in $\R^d$ with $d\le 2$.

The following lemma is due to
\cite{Gesztesy:1991, Gesztesy:1995}.

\begin{lemma}\label{Gesztesy and Zhao} Let $d\le 2$, $L=\Delta/2$ and
$\beta\in {\bf K}^\infty_d(\xi)$. The
following conditions are equivalent.
\begin{description}
\item{(a)}  $\beta$ is critical.
\item{(b)} There exists a positive bounded solution to (\ref{harmonicity}).
\end{description}
Moreover, if $\beta$ is critical, then the positive bounded solution $h$
to (\ref{harmonicity}) is unique (up to constant multiples), and $h$ possesses the
following representation:
$$
h(x)=\left\{\begin{array}{lr}h(0)\lim_{r\downarrow 0}\Pi_xe_\beta(T_{B(0, r)}),
&\quad d=2\\
 h(0)\Pi_xe_\beta(T_0), &\quad d=1,\end{array}\right.
$$
where for
every open set $B$, $T_B=\inf\{ t>0; \xi_t\in B\}$ denotes the first
hitting time of $B$, and $T_0=T_{\{0\}}$ denotes the first hitting
time of $\xi$ at the point $0$.  Moreover, $h$ is  bounded away from zero.
\end{lemma}

It follows from the previous lemma  that, in the case $d\le 2$, if
$\lambda_\infty(\beta)>0$, $\beta-\lambda_\infty(\beta)\in
{\bf K}^\infty_d$ and $\beta-\lambda_\infty(\beta)$ is critical, then the
assumption \eqref{bounded-h}  of Theorem \ref{th3} is satisfied.

\bigskip

\begin{remark}
\rm
Let $d\le 2$ and $L=\Delta/2$. Murata proved the following result (see  \cite[Theorem 4.1]{Murata:1984}): If $\beta\sim |x|^{-\rho}$ $(\rho>4)$ as
$|x|\to\infty$ (obviously  $\beta\in {\bf K}^\infty_d$) and $\beta$ is
subcritical,  then there exists a positive solution $h$
to (\ref{harmonicity}) such that
$$
h(x)=\left\{\begin{array}{lr}(2\pi)^{-1}\log\frac{|x|}{2}+
\mathcal{O}(1),&\quad \mathrm{for}\ d=2,\\
(2\pi^{1/4})^{-1}|x|+\mathcal{O}(1),&\quad  \mathrm{for}\ d=1,\end{array}\right.
$$ as
$|x|\to\infty$.
\end{remark}

Thus if $d\le 2$, $L=\Delta/2$,  $\beta-\lambda\in {\bf K}^\infty_d$ and  $\beta-\lambda$ is
subcritical, then there is no positive bounded solution to $(L +\beta
-\lambda )h=0$. In order to deal with the subcritical case, we
need to develop some results on Schr\"odinger semigroups. We believe, that these
results are also of independent interest.

\begin{lemma}\label{sup-sup} Let $d\le 2$, $L=\Delta/2$ and
$\beta\in {\bf K}^\infty_d$. If
$\lambda_\infty(\beta)=0$, then
\begin{equation}\label{sup-t-sup-x}
\sup_{t\ge 0}\sup_{x\in\R^d}\Pi_xe_\beta(t)<\infty.
\end{equation}
\end{lemma}

\pf  Since
$\lambda_\infty(\beta)=0$, $\beta$ is either critical or subcritical. For the
subcritical case we will prove a  stronger result later,
see Lemma \ref{sup-sup2}  below.
Thus, we now assume that $\beta$ is critical. Then
Lemma \ref{Gesztesy and Zhao} asserts that there exists a bounded
solution $\psi$ to (\ref{harmonicity}) such that $\psi>0$ and
$\sup_{x\in\R^d}\psi^{-1}(x)<\infty$. We then have
$$\begin{array}{rl}\Pi_xe_\beta(t)=&\Pi_x(e_\beta(t)(\psi^{-1}\psi)(\xi_t))\\
\le&\left(\sup_{x\in\R^d}\psi^{-1}(x)\right)\Pi_x(e_\beta(t)\psi(\xi_t))\\
=&\left(\sup_{x\in\R^d}\psi^{-1}(x)\right)\psi(x)\\
\le&\sup_{x\in\R^d}\psi(x)/\inf_{x\in\R^d}\psi(x)<\infty.\end{array}
$$
This proves \eqref{sup-t-sup-x}.
\qed

\begin{remark}
\rm
Murata (see
\cite[Corollary 1.6]{Murata:1984}) proved the above result for  $d=2$ under the
condition that $\beta\sim |x|^{-\rho}$ $(\rho>4)$ as $|x|\to\infty$,
which implies that $\beta\in {\bf K}^\infty_2$. Our proof above goes along the
line given in the proof of \cite[ Corollary 1.6(ii)]{Murata:1984}.
\end{remark}

\bigskip

 If $\beta$ is subcritical,  we have the following stronger result.
\begin{lemma}\label{sup-sup2}
Let $d\le 2$, $L=\Delta/2$ and $\beta\in {\bf K}_d^\infty$. If  $\beta$ is subcritical, then
\begin{equation}\label{sup-x-t}
\sup_{x\in\R^d}\Pi_x\sup_{0\le t\le\infty}e_\beta(t)<\infty.
\end{equation}
\end{lemma}

\pf We first prove the result for dimension $d=2$. For
$r>0$ we denote the open ball of radius $r$ with center at the
origin and its open exterior by
$$
B_r=\{x\in\R^d, \quad |x|<r\};\quad B^*_r=\{x\in\R^d, \quad
|x|>r\}.
$$
According to  \cite[Proposition 2.2]{Gesztesy:1995},
there exists an $r_0>0$ such that for all $r\ge r_0$ and $x\in
B^*_r$,
\begin{equation}\label{gauge-infty}
\Pi_xe_{\beta^+}(\tau_{B^*_r})\le
2, \quad e^{-1/2}\le\Pi_xe_{\beta}(\tau_{B^*_r})\le 2.
\end{equation}
Choose $r_0$ large enough such that $\mbox{supp}(\mu)\subset B_{r_0}$.
We fix two real numbers $r$ and $R$ with $R>r\ge r_0$. Since
$\beta$ is subcritical, by  \cite[Theorem 2.1]{Gesztesy:1991},
$$
\Pi_xe_{\beta}(\tau_{B_R})<\infty,\quad \forall x\in B_R.
$$
We define
$$S=
\tau_{B_R}+\tau_{B^*_r}\circ\theta_{\tau_{B_R}}.
$$
Put
$$
S_0=0;\quad S_n=S_{n-1}+S\circ\theta_{S_{n-1}},\quad n\ge 1.
$$
In particular, $S_1=S$.
For any $f\in C(\partial B_r)$, we define
$$
(A_Sf)(x)=\Pi_x(e_\beta(S)f(\xi_S)),\quad x\in\partial B_r.
$$
Note that
$$
A^n_Sf(x)=\Pi_x\left[e_\beta(S_n)f(S_n)\right], \quad x\in\partial B_r.
$$
The spectral radius of $A_S$ is defined by
$$
\~{\lambda}(\beta):=\lim_{n\to\infty}\|A^n_S\|^{1/n}.
$$
It follows from \cite[Theorem 2.4]{Gesztesy:1995} that $\~{\lambda}(\beta)<1.$
Thus there exists $\delta>0$ such that $\~{\lambda}(\beta)+\delta<1,$
and  sufficiently large $n$ such that,
$\|A^n_S\|\le (\~{\lambda}(q)+\delta)^n.$
Therefore we have
\begin{equation}\label{series-finit}
\sum^{\infty}_{n=0}\sup_{x\in\R^d}|A^n_S1(x)|=\sum^{\infty}_{n=0}
\sup_{x\in\R^d}\Pi_{x}e_\beta(S_n)<\infty.
\end{equation}
By the strong Markov property applied at $\tau_{B_{R}}$, along with the simple fact that $\int^t_0e_{\beta^+}(s)\beta^+(s)\, \mathrm{d}s=e_{\beta^+}(t)-1$,
and finally by\eqref{gauge-infty}, we have
$$\begin{array}{rl}
\Pi_x\displaystyle\int^{S}_0e_\beta(t)\beta^+(t)\, \mathrm{d}t=&
\Pi_x\displaystyle\int^{\tau_{B_R}}_0e_\beta(t)\beta^+(t)\, \mathrm{d}t+
\Pi_x\left[\Pi_{\xi_{\tau_{B_{R}}}}\displaystyle\int^{\tau_{B^*_r}}_0
e_{\beta}(t)\beta^+(t)\,
\mathrm{d}t\right]\\
\le&\displaystyle\Pi_x\int^{\tau_{B_R}}_0e_\beta(t)\beta^+(t)\, \mathrm{d}t+
\Pi_x\left[\Pi_{\xi_{\tau_{B_{R}}}}\displaystyle\int^{\tau_{B^*_r}}_0
e_{\beta^+}(t)\beta^+(t)\, \mathrm{d}t\right]\\
=&\Pi_x\displaystyle\int^{\tau_{B_R}}_0e_\beta(t)\beta^+(t)\, \mathrm{d}t+
\Pi_x\left[\Pi_{\xi_{\tau_{B_{R}}}}e_{\beta^+}(\tau_{B^*_r})\right]-1\\
\le&\Pi_x\displaystyle\int^{\tau_{B_R}}_0e_\beta(t)\beta^+(t)\, \mathrm{d}t+1.\end{array}
$$

Let $\xi^{B_R}$ denote the  Brownian motion killed upon exiting $B_R$.
Since $\beta$ is subcritical,
the function $x\to \Pi_x e_\beta(\tau_{B_R})$ is bounded on $B_R$.
It follows from \cite[Theorem 2.8]{Chen:2002} that
$$
\sup_{x\in B_R}\Pi_x\int^{\tau_{B_R}}_0e_\beta(t)\beta^+(t)\, \mathrm{d}t<\infty.
$$
Thus
\begin{equation}\label{strong-gauge}
C:=\sup_{x\in\partial B_r}\Pi_x\int^{S}_0e_\beta(t)\beta^+(t)\, \mathrm{d}t<\infty.
\end{equation}
By the strong Markov property, applied at $S_n$, and by \eqref{series-finit}, and
\eqref{strong-gauge}, we have
\begin{equation}
\begin{array}{rl}\label{strong-gauge2}
\displaystyle\sup_{x\in\R^d}\Pi_x\int^{\infty}_0e_\beta(t)\beta^+(t)\, \mathrm{d}t
\le&\sum^{\infty}_{n=0}\sup_{x\in\R^d}\Pi_x\left[\displaystyle
\int^{S_{n+1}}_{S_n}e_\beta(t)\beta^+(t)\, \mathrm{d}t\right]\\
=&\sum^{\infty}_{n=0}\sup_{x\in\R^d}\Pi_x\left[e_\beta(S_n)\Pi_{\xi_{S_n}}
\displaystyle\int^{S}_0e_\beta(t)\beta^+(t)\, \mathrm{d}t\right]\\
\le&C\sum^{\infty}_{n=0}\sup_{x\in\R^d}\Pi_xe_\beta(S_n)<\infty.
\end{array}
\end{equation}
Observe that
$$
e_\beta(t)=1+\int^t_0e_\beta(s)\beta(s)\, \mathrm{d}s\le1+\int^t_0
e_\beta(s)\beta^+(s)\, \mathrm{d}s,
$$
and so
$$
\sup_{0\le t\le \infty}e_\beta(t)\le 1+\int^{\infty}_0e_\beta(s)\beta^+(s)\, \mathrm{d}s.
$$
Using \eqref{strong-gauge2} we get \eqref{sup-x-t} and we finish the
proof for dimension $d=2$.

Now let $d=1$. Define
$$u(a, b):=\Pi_xe_\beta(T_b),\quad a, b\in \R^1,$$
where $T_b$ is the first hitting time of $\xi$ at the point $b$. By
\cite[Theorem 4.8]{Gesztesy:1991}, $u(a, b)u(b,a)<1$ for any $a,
b\in\R^1$. For any $x\in \R^1$, define
$$
S_x=T_{x+1}+T_x\circ\theta_{T_{x+1}}.
$$
Then
$$
\Pi_xe_\beta(S_x)=u(x, x+1)u(x+1,x)<1.
$$
Repeating the above proof for $d=2$ with $S$ replaced by $S_x$
we can similarly obtain \eqref{sup-x-t} for $d=1$. We omit the
details.
\qed

\begin{lemma}\label{zero gauge}
Let $d\le 2$, $L=\Delta/2$ and $\beta\in {\bf K}_d^\infty$. If $\beta$ is subcritical, then
\begin{equation}\label{gauge-0}
\lim_{t\to\infty}\Pi_xe_\beta(t)=\Pi_{x}e_\beta(\infty)\equiv
0\quad\mbox{ in } \R^d.
\end{equation}
\end{lemma}

\pf
By \eqref{sup-x-t} and by dominated convergence, it suffices to
show
\begin{equation}\label{gauge0}
\Pi_{x}e_\beta(\infty)= 0,\quad\forall x\in \R^d.
\end{equation}
We continue to use the notations in the proof of  Lemma \ref{sup-sup2}.
We first prove \eqref{gauge0} for dimension $d=2$.
Using the strong Markov property of $\xi$, applied at $\tau_{B_{R}}$,
and Fatou's lemma, we get
$$
\begin{array}{rl}\Pi_{0}e_\beta(\infty)=&\Pi_0\left[e_\beta(\xi_{\tau_{B_r}})
\Pi_{\xi_{\tau_{B_r}}}e_\beta(\infty)\right]\\
\le&\Pi_{0}\left[e_\beta(\tau_{B_r})
\lim_{n\to\infty}|(A_S^n)1(\xi_{\tau_{B_r}})|\right]\\
\le&[\Pi_{0}e_\beta(\tau_{B_r})]\lim_{n\to\infty}\|A_S^n\|\\
\le&\left[\Pi_0e_\beta(\tau_{B_r})\right]\lim_{n\to\infty}
(\~{\lambda}(\beta)+\delta)^n=0.
\end{array}
$$
Thus by Lemma \ref{Gauge-positive} in the Appendix,
 $\Pi_{x}e_\beta(\infty)\equiv 0$ in
$\R^2.$

Now let $d=1$. For any $x\in\R$, let $S_x$ be defined as in
proof of Lemma \ref{sup-sup2}. By the strong Markov property of
$\xi$ applied at $S_x$, we have, for any $x\in\R$,
$$
\Pi_xe_\beta(\infty)=\Pi_xe_\beta(S_x)\Pi_xe_\beta(\infty).
$$
Since $\Pi_xe_\beta(S_x)=u(x, x+1)u(x+1,x)<1$, the above equality yields
$\Pi_xe_\beta(\infty)=0$ for every $x\in\R$.
\qed

It follows from the two results above that, if $d\le 2$, $L=\Delta/2$, $\lambda_\infty(\beta)>0$,
$\beta-\lambda_\infty(\beta)\in {\bf K}^\infty_d$ and $\beta-\lambda_\infty(\beta)$ is subcritical,
then the assumptions of Theorem \ref{th2}(2) are satisfied.

\subsection{Compactly supported mass annihilation}

We conclude
with two simple examples which satisfy the assumptions of Theorem \ref{th2}(2).
In both cases we consider compactly supported mass annihilation terms.

We start with a two-dimensional example.
\begin{example}[d=2; constant annihilation in a compact set]\label{planar.BM.example}\rm
Let  $\xi$  be planar Brownian motion, and $\beta(x):=-\alpha \mathbf{1}_{K}(x)$
with $\alpha>0$ being a constant and $K\subset\R^2$ a compact with non-empty interior.
\end{example}
\begin{proposition}\label{p:5.12}
In this case weak extinction holds, whatever $k$ is.
\end{proposition}
\begin{remark}\rm  The point is that our result is true for {\it any} $k$. Indeed, it is easy
to show that extinction holds when $k$ is bounded from below
(even with $\beta\equiv 0$).
\end{remark}
\begin{proof}
It is well known that  $\beta$ is
subcritical (see, e.g., \cite[Theorem 1.4]{Murata:1984}). By
\cite[Corollary 2]{Chan:1994}, as $t\to\infty$,
$$
\Pi_x\left[\exp\left(\int^t_0\beta(\xi_s)\, \mathrm{d}s\right)\right]\sim c(\log t)^{-1},
$$
where $0<c=c(x,K,\alpha)$. Therefore, for any $x\in \R^2$, $\lambda_\infty(\beta)\ge
\lim_{t\to\infty}\frac{1}{t}\log\Pi_x e_\beta(t)=0$. It is obvious
that $\lambda_\infty(\beta)\le 0$. Then $\lambda_\infty =0$ and
$g_{\beta-\lambda_\infty }(x)\equiv 0$. Clearly,
\eqref{upper-domi} holds since $\beta\le 0$. Using again that $\beta\le 0$, we are done
by part (2) of Theorem \ref{th2}. \qed
\end{proof}

Finally, we discuss an example in one-dimension.
\begin{example}[d=1; compactly supported mass annihilation]\rm  Let  $\xi$ be a
Brownian motion in $\R$, and $\beta\le 0$  a continuous
function on $\R$ with compact support.
\end{example}
\begin{proposition}
Again, weak extinction  holds, whatever $k$ is.
\end{proposition}
\begin{proof}
It is well known that
$\beta$ is subcritical (see \cite{Reed:1978}).
By
\cite{Yamada:1986},
\begin{equation}\label{limit-beta}
\lim_{t\to\infty}t^{-1/2}\int^t_0\beta(\xi_s)\, \mathrm{d}s=
\eta\int_{-\infty}^{\infty}\beta(x)\mathrm{d}x,
\end{equation}
in distribution, where $\eta$ is a random variable with $\eta< 0$ a.s. This, along with
Jensen's inequality, implies that, abbreviating
$a:=\int_{-\infty}^{\infty}\beta(x)\, \mathrm{d}x,$
$$
\liminf_{t\to\infty}\left[\Pi_x\exp\left(\int^t_0\beta(\xi_s)\, \mathrm{d}s\right)\right]^{t^{-1/2}}\ge\lim_{t\to\infty}\Pi_x
\exp\left(t^{-1/2}\int^t_0\beta(\xi_s)\, \mathrm{d}s\right)=\Pi_x\exp(a\eta).
$$
 Hence,
$$
\liminf_{t\to\infty}t^{-1/2}\log\left[\Pi_x\exp\left(\int^t_0\beta(\xi_s)\, \mathrm{d}s\right)\right]\ge\log\Pi_x\exp(a\eta).
$$
Thus, for $f(t):=t^{-1}\log\Pi_x\exp\left(\int^t_0\beta(\xi_s)\, \mathrm{d}s\right)$, we have
$\liminf_{t\to\infty}f(t)\ge0.$
But $\beta\le 0$ implies that $\limsup_{t\to\infty}f(t)\le 0$, and so
$
\lambda_\infty =\lim_{t\to\infty}f(t)=0.
$
By \eqref{limit-beta} (or, by the recurrence of $\xi$),
$
g_{\beta-\lambda_\infty }(x)=\Pi_x\exp\left(\int^{\infty}_0
\beta(\xi_s)\, \mathrm{d}s\right)\equiv 0.
$
Again, $\beta\le 0$ implies \eqref{upper-domi}, and we finish as in the proof of
Proposition \ref{p:5.12}.
\qed
\end{proof}

\section{Proofs of Theorem \ref{th1} and Theorem \ref{th2}}

For any nonzero $\mu\in {M}(\R^d)$, define
\begin{equation}
\Pi_{\mu}=\int_D\Pi_{x}\,\mu(\mathrm{d}x).
\end{equation}

The following result is \cite[Lemma 1.5]{Dynkin:1993}.

\begin{lemma}\label{rewrite-int}
The equation \eqref{inhom-int} is equivalent to
\begin{equation}\label{inhom-int2}
u(t,x)+ \Pi_{x}\int^{t}_0e_{\beta}(s)k(\xi_s)(u(t-s,
\xi_s))^2\mathrm{d}s=\Pi_{x}(e_{\beta}(t)f(\xi_{t})).
\end{equation}
Moreover, $u$ is the minimal non-negative solution to \eqref{inhom-int} if and only if $u$ is the minimal non-negative solution to \eqref{inhom-int2}.
\end{lemma}

Combining \eqref{inhom-fund} and \eqref{inhom-int2}, we get
the following expectation and variance formulae: for any bounded nonnegative function
$f$ on $\R^d$ and any nonzero $\mu\in M(\R^d)$,
\begin{equation}\label{Exp}
\P_{ \mu}\langle f, X_{t}\rangle
=\Pi_{\mu}(f(\xi_{t})e_{\beta}(t))
\end{equation}
and
\begin{equation}\label{variances}
\mbox{Var}_{ \mu}\langle f, X_{t}\rangle
=\Pi_{ \mu}\left(\int^{t}_{0}e_{\beta}(s)k(\xi_s)
2[\Pi_{\xi_s}e_{\beta}(t-s)f(\xi_{t-s})]^2\mathrm{d}s\right),
\end{equation}
where $\mbox{Var}_{\mu}$ stands for variance under $\P_{\mu}$.

\begin{lemma}\label{domi-g(t)} If $\lambda_\infty>0$, then
\begin{equation}
\liminf_{t\to\infty}\|P^\beta_t1\|_\infty^{-1}\int^t_0\|P^\beta_s1\|_\infty \, \mathrm{d}s<\infty.
\end{equation}
\end{lemma}

\begin{proof}
For convenience, we denote $\|P^\beta_t1\|_\infty$ by $h(t)$ in this proof.
Suppose that the statement is false. Then
$$
\lim_{t\to\infty}\frac{\int^t_0h(s) \, \mathrm{d}s}{h(t)}=\infty,
$$
and so for any $K>0$, there exists $T_K>0$ such that for $t>T_K$,
$$
\frac{\int^t_0h(s) \, \mathrm{d}s}{h(t)}>K,
$$
i.e.,
$$
h(t)<\frac{1}{K}\int^t_0h(s)\, \mathrm{d}s=\alpha+\frac{1}{K}\int^t_{T_K}h(s)\, \mathrm{d}s,
$$
where $\alpha=\frac{1}{K}\int^{T_K}_0h(s)\, \mathrm{d}s$.
By Gronwall's lemma, we get
$$
h(t)\le \alpha \left(e^{(t-T_2)/K}-1\right).
$$
However, if $\frac1{K}<\frac{\lambda_\infty}{2}$ ($K>\frac{2}{\lambda_\infty}$),
then this contradicts the following easy consequence of the definition \eqref{def-Lambda} of $\lambda_\infty$:
$$
\lim_{t\to\infty}\frac{\log h(t)}{t}\ge \frac{\lambda_\infty}{2}.
$$
This contradiction proves the lemma.
\qed
\end{proof}
\subsection{Proof of Theorem \ref{th1}}
For the proof of the theorem, we will need the following slight generalization of
Doob's maximal inequality for submartingales.

\begin{lemma}\label{modified.Doob}
Assume that $T\in (0,\infty)$, and that the
non-negative, right continuous,
adapted  process $(\{M_t\}_{0\le t\le T},\{\mathcal{F}_t\}_{0\le t\le T},\mathbf P)$
satisfies that there exists an $a>0$  such that
$$\mathbf P(M_t\mid \mathcal{F}_s)\ge a M_s,\  0\le s<t\le T.$$
Then, for every $\alpha\in(0,\infty)$ and $0\le S\le T$,
 $$
 \mathbf{P}\left(\sup_{t\in [0,S]}M_t\ge \alpha\right)\le (a\alpha)^{-1}\mathbf P[M_S].
 $$

\end{lemma}
\begin{proof} Looking at the proof of Doob's inequality
(see \cite[Theorems 5.2.1 and 7.1.9]{Stroockbook} and their proofs),
one can see that, when the submartingale property is replaced by our assumption,
the whole proof goes through, except that now one has to include a factor
 $a^{-1}$ on the right hand side.$\hfill\square$
\end{proof}

\bigskip
\noindent{\bf Proof of Theorem \ref{th1}:}
(1) By a standard Borel-Cantelli argument, it suffices to prove that with an appropriate
choice of $T>0$, it is true that for any given $\epsilon>0$,
\begin{equation}\label{B-C}
\sum_n \P_\mu\left(\sup_{s\in[0,T]}
e^{-\lambda (nT+s)}\|X_{nT+s}\|>\epsilon\right)<\infty.
\end{equation}
Pick
\begin{equation}
\gamma \ge  -\lambda.\label{gamma.is.large}
\end{equation}
Then
\begin{equation}\label{first.inequality}
\P_\mu\left(\sup_{s\in[0,T]}e^{-\lambda (nT+s)}\|X_{nT+s}\|>\epsilon\right)\le
\P_\mu\left(\sup_{s\in[0,T]}e^{\gamma (nT+s)}\|X_{nT+s}\|>\epsilon\cdot
e^{(\lambda+\gamma)nT}\right).
\end{equation}
Let $M^{(n)}_t:=e^{\gamma (nT+t)}\|X_{nT+t}\|$ for $t\in[0,T].$
Pick a number $0<a<1$ and fix it.
Let $\mathcal{F}^{(n)}_{s}:=\sigma(X_{nT+r}: r\in [0, s])$.
If we show that for a sufficiently small $T>0$ and all $n\ge 1$, the process
$\{M^{(n)}_t\}_{0\le t \le T}$
satisfies that for all $0<s<t<T$,
\begin{equation}\label{almost.submartingale}
\mathbb P_{\mu}(M^{(n)}_t\mid \mathcal{F}^{(n)}_s)\ge a M^{(n)}_s,
\end{equation}
then, by using  Lemma \ref{modified.Doob},  we can continue (\ref{first.inequality}) with
\begin{eqnarray*}
\P_\mu\left(\sup_{s\in[0,T]}e^{-\lambda (nT+s)}\|X_{nT+s}\|>\epsilon\right)
&\le&
\frac{1}{a\epsilon} e^{-(\lambda+\gamma)nT}\P_{\mu}
\left[e^{\gamma(n+1)T}\|X_{(n+1)T}\|\right]
\\
&=&\frac{1}{a\epsilon} e^{(\lambda+\gamma)T}e^{-\lambda(n+1)T}
\P_{\mu} \|X_{(n+1)T}\|\\
&\le&\frac{\|\mu\|}{a\epsilon} e^{(\lambda+\gamma)T}
e^{-\lambda(n+1)T}\|P^{\beta}_{(n+1)T}1\|_{\infty}.
\end{eqnarray*}
Since $\lambda>\lambda_{\infty}$ and
$\|P^{\beta}_{(n+1)T}1\|_{\infty}=\exp(\lambda_{\infty}(n+1)T+o(n))$ as $n\to\infty$,
therefore (\ref{B-C}) holds.

It remains to check (\ref{almost.submartingale}).  Let $0<s<t<T$. Using the Markov and
branching properties at
time $nT+s$,
\begin{eqnarray}\label{long}
\nonumber \P_{\mu}\left[M^{(n)}_t\mid \mathcal{F}^{(n)}_s\right]
&=&\P_{X_{nT+s}}e^{\gamma(nT+t)}\|X_{t-s}\|=
\left\langle \P_{\delta_{x}} e^{\gamma(nT+t)}\|X_{t-s}\|,\
X_{nT+s}(\mathrm{d}x)\right\rangle \\
&=&\left\langle \P_{\delta_{x}} e^{\gamma(t-s)}\|X_{t-s}\|,\
e^{\gamma(nT+s)}X_{nT+s}(\mathrm{d}x)\right\rangle.
\end{eqnarray}
At this point we are going to determine $T$ as follows.
According to the assumption  $\beta\in\mathbf{K}(\xi)$, $$\lim_{t\downarrow 0}
\sup_{x\in \R^d}\Pi_x\int_0^t |\beta| (\xi_s)\, \mathrm{d}s=0.$$  Pick $T>0$ such that
$$
\gamma t +\Pi_x \int_0^t \beta (\xi_s)\, \mathrm{d}s\ge \log a,
$$
for all $0<t< T$ and all $x\in \R^d$.
By Jensen's inequality,
$$
\gamma t+\log\Pi_x \exp\left(\int_0^t \beta (\xi_s)\, \mathrm{d}s\right)\ge \log a,
$$
and thus
$$
\mathbb P_{\delta_{x}}e^{\gamma t}\|X_t\|=e^{\gamma t}\Pi_x \exp\left(\int_0^t
\beta (\xi_s)\, \mathrm{d}s\right)\ge a
$$
holds too, for all $0<t< T$ and all $x\in \R^d$.
Returning to (\ref{long}), for $0<s<t<T$,
$$
\P_{\mu}[M^{(n)}_t\mid \mathcal{F}^{(n)}_s]\ge a \left \langle 1,\
 e^{\gamma(nT+s)}X_{nT+s}\right\rangle=aM^{(n)}_s,\ a.s.,
$$
yielding (\ref{almost.submartingale}).

(2)  First note that to prove \eqref{limsup=infty2} it suffices
to prove that there exists $c_0>0$ such that for all $K>0$,
\begin{equation}
\label{limsup>K}
\P_{\mu}\left(\limsup_{t\to\infty}e^{-\lambda t}
\|X_t\|\ge K\right)\ge c_0.
\end{equation}
Since
$$
\left\{\limsup_{t\to\infty}e^{-\lambda t}\|X_t\|\ge K\right\}
\supseteq\limsup_{t\to\infty}\{e^{-\lambda t}\|X_t\|\ge K\},
$$
we have by the reverse Fatou lemma,
\begin{equation}\label{domi-below1}\begin{array}{rl}
\P_{\mu}\left(\displaystyle\limsup_{t\to\infty}e^{-\lambda t}\|X_t\|\ge
K\right)&\ge\displaystyle\limsup_{t\to\infty}\P_{\mu}(e^{-\lambda t}\|X_t\|\ge
K)\\&=\displaystyle\limsup_{t\to\infty}\P_{\mu}(e^{-\lambda t}\|X_t\|-K\ge 0).\end{array}
\end{equation}
The assumption
$\lambda<\lambda_\infty$ implies that
 \begin{equation}\label{limit0}
\lim_{t\to\infty}\P_{\mu}(e^{-\lambda t}\|X_t\|)=\lim_{t\to\infty}e^{-\lambda
 t}\Pi_{\mu}e_{\beta}(t)=\infty.
 \end{equation}
Thus $\P_{\mu}e^{-\lambda t}\|X_t\|>K$ for large $t$.
It  follows easily from the Cauchy-Schwarz inequality (see, for instance,
\cite[Chap. 1, Ex. 3.8]{Durrett:1996}) that
for any nonnegative random variable $Y$ with finite second moment, on a probability
space $(\Omega, {\cal G}, P)$, and for any $a>0$,
$$
P(Y-a\ge 0)\ge \frac{(PY -a)^2}{P(Y^2)}.
$$
Applying the above inequality (`Paley-Zygmund inequality') with
$Y=e^{-\lambda t}\|X_t\|$ and $a=K$, we get
\begin{equation}\label{domi-below2}
\P_{\mu}(e^{-\lambda t}\|X_t\|-K\ge 0)\ge\frac{\left(\P_{\mu}
e^{-\lambda t}\|X_t\|-K\right)^2}{\P_{\mu}(e^{-\lambda t}\|X_t\|)^2}.
\end{equation}
By  \eqref{Exp} and \eqref{variances}, \eqref{domi-below1} and
\eqref{domi-below2} yield
\begin{equation}\label{domi-below3}
\begin{array}{rl}&\P_{\mu}
(\displaystyle\limsup_{t\to\infty}e^{-\lambda t}\|X_t\|\ge K)\\\ge&
\displaystyle\limsup_{t\to\infty}\displaystyle
\frac{(\Pi_{\mu}e^{-\lambda t}e_{\beta}(t)-K)^2}
{(\Pi_{\mu}e^{-\lambda t}e_{\beta}(t))^2+
2e^{-2\lambda t}\Pi_{\mu}\displaystyle\int^t_0e_{\beta}(s)k(\xi_s)
[\Pi_{\xi_s}e_{\beta}(t-s)]^2\, \mathrm{d}s}\\
=&\displaystyle\limsup_{t\to\infty}\left(1-K\frac{e^{\lambda
t}}{\Pi_{\mu}e_{\beta}(t)}\right)^{2}\left(1+2\displaystyle
\frac{\Pi_{\mu}\left(e_{\beta}(t)\int^t_0k(\xi_s)\Pi_{
\xi_s}e_{\beta}(t-s)\, \mathrm{d}s\right)}{ (\Pi_{\mu}e_{\beta}(t))^2}
\right)^{-1}\\
=& \displaystyle\limsup_{t\to\infty}\left(1+2\displaystyle
\frac{\Pi_{\mu}\left(e_{\beta}(t)\int^t_0k(\xi_s)\Pi_{
\xi_s}e_{\beta}(t-s)\, \mathrm{d}s\right)}{ (\Pi_{\mu}e_{\beta}(t))^2}
\right)^{-1}.\end{array}
\end{equation}
Note that
$$
\Pi_{\xi_s}e_{\beta}(t-s)\le \|P^{\beta}_{(t-s)}1\|_\infty.
$$
Thus we have
$$
\begin{array}{rl}\Pi_{\mu}\left(e_{\beta}(t)\displaystyle\int^t_0k(\xi_s)\Pi_{
\xi_s}e_{\beta}(t-s)\, \mathrm{d}s\right) \le&
\|k\|_\infty\Pi_{\mu}e_{\beta}(t)\left[\displaystyle\int^t_0
\|P^{\beta}_{t-s}1\|_\infty \, \mathrm{d}s\right]\\
=&\|k\|_\infty\Pi_{\mu}e_{\beta}(t)\left[\displaystyle\int^t_0
\|P^{\beta}_{s}1\|_\infty \, \mathrm{d}s\right]
.\end{array}
$$
So, we have for every $K>0$,
\begin{equation}\label{d3}
\P_{\mu}
(\displaystyle\limsup_{t\to\infty}e^{-\lambda t}\|X_t\|\ge K)\ge\left(1+2\liminf_{t\to\infty}
\frac{\|k\|_\infty\|P^{\beta}_t1\|_\infty^{-1}\int^t_0\|P^{\beta}_s1\|_\infty \, \mathrm{d}s}{
\|P^{\beta}_t1\|_\infty^{-1}\Pi_{\mu}e_{\beta}(t)}\right)^{-1}.
\end{equation}

We now consider the numerator and denominator of the right-hand side of \eqref{d3} separately.
$$
\liminf_{t\to\infty}\|P^{\beta}_t1\|_\infty^{-1}\int^t_0\|P^{\beta}_s1\|_\infty \, \mathrm{d}s<\infty.
$$
By Fatou's lemma and  \eqref{uniform-bd},
$$
\liminf_{t\to\infty}\|P^{\beta}_t1\|_\infty^{-1}\Pi_{\mu}e_{\beta}(t)\ge\left\langle
\mu, \liminf_{t\to\infty}\|P^{\beta}_t1\|_\infty^{-1}\Pi_{\cdot}e_{\beta}(t)\right\rangle>0.
$$
Now combining \eqref{d3} and Lemma \ref{domi-g(t)}, we arrive at
\eqref{limsup>K}.
\qed
\subsection{Proof of Theorem \ref{th2}}

(1)
Using Fatou's lemma and \eqref{g(t)-to-infty},
we get
$$\liminf_{t\to\infty}e^{-\lambda_\infty
t}\Pi_{\mu}e_{\beta}(t)=\liminf_{t\to\infty}\Pi_{\mu}
e_{\beta-\lambda_\infty}(t)\ge\left\langle
\liminf_{t\to\infty}\Pi_{\cdot}e_{\beta-\lambda_\infty}(t),
\mu\right\rangle=\infty,$$
which means that \eqref{limit0} holds with $\lambda$ replaced by
$\lambda_\infty$. So the proof of Theorem \ref{th1}(2) works with $\lambda$
replaced by $\lambda_\infty$.

(2)
By \eqref{Exp}, we have
\begin{equation}\label{expection}
\P_{\mu}[\exp(-\lambda_\infty t)\|X_t\|]=
\Pi_{\mu}e_{\beta-\lambda_\infty}(t).
\end{equation}
Letting $t\to\infty$ and using Fatou's lemma, we get
\begin{equation}\label{liminf-expec}
\P_{\mu}(\liminf_{t\to\infty}\exp(-\lambda_\infty t)\|X_t\|)\le \liminf_{t\to\infty}
\Pi_{\mu}e_{\beta-\lambda_\infty}(t).
\end{equation}
Note that $\Pi_{\mu}e_{\beta-\lambda_\infty}(t)=\langle
\Pi_{\cdot}e_{\beta-\lambda_\infty}(t), \mu\rangle$.
Using \eqref{upper-domi} and the assumption that $g_{\beta-\lambda_\infty}\equiv 0$
in $\R^d$, we get
$$
\lim_{t\to \infty}\Pi_{\mu}e_{\beta-\lambda_\infty}(t)=\langle
\lim_{t\to\infty}\Pi_{\cdot}e_{\beta-\lambda_\infty}(t), \mu\rangle=\langle
g_{\beta-\lambda_\infty}, \mu\rangle=0,
$$
where in the first equality we used the fact
$\Pi_{\cdot}e_{\beta-\lambda_\infty}(t)\le
\sup_{x\in\R^d}\Pi_{x}(\sup_{t\ge 0}e_{\beta-\lambda_\infty }(t))<\infty$,
which follows from \eqref{upper-domi}, and the fact that $\mu$ is finite measure,
and in the second equality we used the fact $e_{\beta-\lambda_\infty}(t)
\le\sup_{t\ge 0}e_{\beta-\lambda_\infty}(t)<\infty$ $\Pi_x$-a.s.
for any $x\in\R^d$.
Hence by
\eqref{liminf-expec} we get
$$
\P_{\mu}\left(\liminf_{t\to\infty}\exp(-\lambda_\infty t)\|X_t\|=0\right)=1,
$$
which implies  \eqref{liminf=0}.

Finally, when $\beta\le 0$, trivially $\lambda_{\infty}\le 0$; hence
$
\P_{\mu}(\liminf_{t\to\infty}\|X_t\|=0)=1.
$
On the other hand, $\|X\|$ is a supermartingale by the expectation formula and
the branching Markov property, and thus, $\lim_{t\to\infty}\|X_t\|$ exists $\P_{\mu}$-a.s.
Hence, we can improve the liminf to a limit. \qed

\section {\bf Proofs of  Theorems \ref{th3} and  \ref{main2}}

\subsection{Proof of Theorem \ref{th3}}
We start with a lemma.
\begin{lemma}\label{martingale}  Assume that $\beta\in {\bf K}(\xi)$ and
that $h>0$ is a bounded solution to
$$
(L+\beta-\lambda_\infty)h=0\ \text{in}\
\R^d
$$
in the sense of distributions.
Let  $\mu\in {M}(\R^d)$ be nonzero and ${\cal F}_t
:=\sigma\{X_r, r\le t\}$. Then  the process
$(\{e_{-\lambda_\infty}(t)\langle h, X_t\rangle \}_{t\ge 0},
 \{  \mathcal{F}_{t} \}_{t\ge 0}
 ,\P_{\mu})$
is a positive martingale.
\end{lemma}

\pf
Recall that $D_n=B(0, n)$ and $\tau_n$ is the first exit time of $\xi$ from $D_n$.
Since $h$ is harmonic with respect to the operator $L+\beta-\lambda_\infty$, we have
\begin{equation}\label{harmonic}
h(x)=\Pi_{x}\left[e_{\beta-\lambda_\infty}(t\wedge \tau_n)
h(\xi_{t\wedge \tau_n})\right],\quad\mbox{ for
every } n\ge 1 \mbox{ and } t\ge 0,
\end{equation}
see the proof of \cite[Lemma 2.1]{SoV}. Since $h$ is bounded,  bounded convergence yields \begin{equation}\label{harmonic2}
h(x)=\Pi_{x}\left[e_{\beta-\lambda_\infty}(t)h(\xi_t)\right],\quad\mbox{ for
every }t\ge 0.
\end{equation}
By the branching and Markov properties, for $r\le s<t$,
we have
\begin{equation}
\begin{array}{rl}&\P_{\mu}(e_{-\lambda_\infty}(t)\langle h, X_t\rangle|{\cal
F}_s)\\=& e_{-\lambda_\infty}(t)\P_{ X_s}\langle h, X_{t-s}\rangle\\
=&e_{-\lambda_\infty}(t)\langle\Pi_{
\cdot}\left(e_{\beta}(t-s)h(\xi_{t-s})\right),
X_s\rangle\\
=&e_{-\lambda_\infty}(t)\langle\Pi_{\cdot}\left(e_{\beta}(
t-s)h(\xi_{t-s})\right),
X_s\rangle\\
=&e_{-\lambda_\infty}(s)\langle h,
X_s\rangle,\end{array}
\end{equation}
proving the assertion. \qed

\noindent{\bf Proof of Theorem \ref{th3}:} Suppose $\mu\in M_c(\R^d)$. Since $M^h$ defined by $$M^h_t:=\exp(-\lambda_\infty t)\langle h, X_t\rangle$$  is a nonnegative $\P_\mu$-martingale,
 $\lim_{t\to\infty}M^h_t$ exists and is also finite $\P_{\mu}$-a.s.
By the martingale property, we have
$$
\P_{\mu}M_t^h=\exp(-\lambda_\infty
t)\Pi_{\mu}[e_{\beta}(t)h(\xi_t)] =\langle h,
\mu\rangle.
$$
It follows from \eqref{<infty2} and Lemma \ref{inequ-green} in  Appendix B that
$$
\Pi_{\mu}\left[\displaystyle\int^\infty_0e_{\beta-2\lambda_\infty}(s)k(\xi_s)
h^2(\xi_{s}))\, \mathrm{d}s \right]\le C^2 \Pi_{\mu}\left[\displaystyle
\int^\infty_0e_{\beta-2\lambda_\infty}(s)k(\xi_s)
)\, \mathrm{d}s \right]<\infty,
$$ where $C$ is a positive constant such that $h(x)\le C$ for all $x\in \R^d$.
Thus by the variance formula \eqref{variances} and by
\eqref{harmonic}, we have
$$
\begin{array}{rl}&\P_{\mu}\left[M_t^h\right]^2\\
=&\langle h, \mu\rangle^2+\exp(-2\lambda_\infty
t)\Pi_{\mu}\left[\displaystyle\int^t_0
e_{\beta}(s)k(\xi_s)[\Pi_{\xi_s}(e_{\beta}(t-s)h(\xi_{t-s}))]^2\, \mathrm{d}s\right]\\
=&\langle h,
\mu\rangle^2+\Pi_{\mu}\left[\displaystyle\int^t_0
e_{\beta}(s)\exp(-2\lambda_\infty
s)k(\xi_s)[\Pi_{\xi_s}(e_{\beta-\lambda_\infty}(t-s)h(\xi_{t-s}))]^2\, \mathrm{d}s\right]
\\=&\langle h,
\mu\rangle^2+\Pi_{\mu}\left[\displaystyle\int^t_0e_{\beta-2\lambda_\infty}(s)k(\xi_s)
h^2(\xi_{s}))\, \mathrm{d}s \right].\end{array}
$$
By  the $L^2$-convergence theorem,
$M_t^h$ converges to some
$\eta$  in $L^2(\P_{\mu})$. In particular,
$$
0<\P_{\mu}\eta^2=\langle h,
\mu\rangle^2+\Pi_{\mu}\displaystyle\int^{\infty}_0e_{\beta-2\lambda_\infty}(s)k(\xi_s)
h^2(\xi_{s}))\, \mathrm{d}s<\infty,
$$
and therefore,
\begin{equation}\label{eta}
\P_{\mu}\left(\eta<\infty\right)=1,
\quad
\text{and}
\quad
\P_{\mu}\left(\eta=0\right)<1.
\end{equation}
It is obvious that $\P_{\mu}\left(\eta=0\right)<1$ implies that $\P_\mu(\|X_t\|>0,\ \forall t>0)>0$.

If $h$ satisfies \eqref{bounded-h}, then \eqref{eta} implies \eqref{limsup=infty} and \eqref{limsup=0}.
\qed

\begin{remark}\rm
Theorem \ref{th3} says that, under condition  \eqref{<infty2}, not only the Kesten-Stigum
Theorem holds ({\it i.e}., the martingale $M^h_t=
e^{-\lambda_\infty t}\langle h, X_t\rangle$ converges  in $L^1(\P_{\mu})$
as $t\to\infty$), but it can be upgraded to convergence in $L^2(\P_{\mu})$.
We plan to find a  necessary and sufficient condition in an upcoming paper.

Using the `spine' method developed in  Engl\"ander and Kyprianou \cite{Englander:2004},
we can  give an alternative proof of Theorem \ref{th3},
but with the weaker conclusion that  the martingale $M^h_t=e^{-\lambda_\infty t}\langle h, X_t\rangle$ converges
in $L^1(\P_{\mu})$ as $t\to\infty$.

\end{remark}

\subsection{Preparation for the proof of Theorem \ref{main2}}
In the remainder of this section, we suppose $\lambda_\infty=0$ and that $h>0$
is a bounded solution to $(L+\beta)u=0$ in $\R^d$ in the sense of distributions. For $c>0$, put
\begin{equation}
u_{ch}(t,x):=-\log
\P_{\delta_x}\exp(-c\langle h, X_t\rangle),
\end{equation}
then $u_{ch}(t,x)$ is a solution of the following integral equation:
\begin{equation}
u_{ch}(t,x)+
\Pi_{x}\int^{t}_0\left[k(\xi_r)\left(u_{ch}(t-r,
\xi_r)\right)^2-\beta(\xi_r)u_{ch}(t-r, \xi_r)\right]\mathrm{d}r=c\Pi_{x}
h(\xi_t).
\end{equation}
By Lemma \ref{rewrite-int}, the above integral equation is equivalent to
\begin{equation}\label{u_ch-equi2}
u_{ch}(t,x)+
\Pi_{x}\int^{t}_0e_{\beta}(r)k(\xi_r)\left[u_{ch}(t-r,
\xi_r)\right]^2\mathrm{d}r=c\Pi_{x}\left[
e_{\beta}(t)h(\xi_t)\right].
\end{equation}
Since $h$ is a bounded positive solution
to $(L+\beta)u=0$, we have
$$
\Pi_{x}\left[e_{\beta}(t) h(\xi_t)\right]
=h(x).
$$
Thus \eqref{u_ch-equi2} can be rewritten as
\begin{equation}\label{u_ch-equi3}
u_{ch}(t,x)+
\Pi_{x}\left[\int^{t}_0e_{\beta}(r)k(\xi_r)\left[u_{ch}(t-r,
\xi_r)\right]^2\mathrm{d}r\right]=ch(x).
\end{equation}
In particular,
\begin{equation}\label{domi}
u_{ch}(t,x)\le ch(x).
\end{equation}
Put
\begin{equation}\label{defu_ch2}
u_{ch}(x):=-\log
\P_{\delta_x}\exp(-c\lim_{t\to\infty}\langle h,
X_t\rangle).
\end{equation}
By Lemma \ref{martingale}, under $\P_{\mu}$,
$\exp(-c\langle h, X_t\rangle)$, $t\ge 0$ is a bounded
submartingale. Thus $u_{ch}(t, x)$ is non-increasing in $t$. Hence,
by the dominated convergence theorem, for every $x\in \R^d$,
$$
u_{ch}(t,x)\downarrow u_{ch}(x)\quad\mbox{ as }t\uparrow\infty.
$$
Note that if $k$ and $\beta$ are radial functions, and  if $L$ is
radial, then $u_{ch}(\cdot)$ is a radial function, i.e.,
$$
u_{ch}(x)=u_{ch}(\|x\|).
$$

\begin{lemma}(1) For any $x\in \R^d$ and $r>0$,
$$
u_{ch}( x)\le\Pi_{x}(u_{ch}(\xi_{\tau_{B(x, r)}})e_{\beta}(
\tau_{B(x, r)})).
$$

(2) If $L$, $k$ and $\beta$ are radial, then
\begin{equation}\label{D-u_ch-above2}
u_{ch}(x)=u_{ch}(\|x\|)\le u_{ch}(R) \Pi_{x}(e_{\beta}(\tau_{B(0,R)})),\quad \|x\|<R.
\end{equation}
\end{lemma}

\pf  (1) By the special Markov
property, for every fixed $x\in\R^d$, one has
$$
\begin{array}{rl}\exp(-u_{ch}(x))=&\P_{\delta_x}
\exp(-c\lim_{t\to\infty}\langle h, X_t\rangle)\\
=&\P_{\delta_x}\left(P_{X_{\tau_{B(x, r)}}}\exp(-c\lim_{t\to\infty}\langle
h, X_t\rangle)\right)\\
=&\P_{\delta_x}\exp\langle -u_{ch},
X_{\tau_{B(x, r)}}\rangle.\end{array}
$$
By Jensen's inequality,
$$
\exp(-u_{ch}(x))\ge \exp(-\P_{\delta_x}\langle u_{ch},
X_{\tau_{B(x, r)}}\rangle)=\exp[-\Pi_{x}(u_{ch}(\xi_{\tau_{B(x, r)}})e_{\beta}(\tau_{B(x, r)}))],
$$
which implies $u_{ch}( x)\le\Pi_{x}(u_{ch}(\xi_{\tau_{B(x, r)}})e_{\beta}(
\tau_{B(x, r)}))$.

(2) Similarly we have, for $x\in B(0, R)$, that
$$
\begin{array}{rl}u_{ch}(
x)\le u_{ch}(R)\Pi_{x}(e_{\beta}(\tau_{B(0,R)})).\end{array}
$$
\qed

Note that $u_{ch}(x)$ is increasing in $c$. Let
\begin{equation}
u_{ch}(x)\uparrow u_{\infty}(x)=
-\log \P_{\delta_x}(\lim_{t\to\infty}\langle h,
X_t\rangle=0).
\end{equation}

\begin{lemma}\label{u-positive}
Either $u_{\infty}(x)\equiv 0$ or $u_{\infty}\in (0,\infty]$ in $\R^d$. That is, if
$$E_h:=\left\{\lim_{t\to\infty}\langle h, X_t\rangle=0\right\},$$
then either
$
\P_{\delta_x}(E_h)= 1,  \forall x\in\R^d,
$
or
$
\P_{\delta_x}(E_h)<1, \forall x\in\R^d.
$
\end{lemma}

\pf We first prove that if there exists a measurable
set $A\subset \R^d$  with positive Lebesgue measure such that
$u_{\infty}>0$ on $A$, then $u_{\infty}(x)>0$ for every $x\in\R^d$.
Indeed, for every $x\in\R^d$,
\begin{equation}\label{P-<h X_t>=02}\begin{array}{rl}
&\P_{\delta_x}(\lim_{t\to\infty}
\langle h, X_t\rangle=0)\\
=&\P_{\delta_x}(\P_{X(1)}(\lim_{t\to\infty}\langle h,
X_t\rangle=0))\\=&\P_{\delta_x}\exp\langle - u_{\infty},
X(1)\rangle.
\end{array}\end{equation}
 Note that
\begin{equation}\label{E-u-X(1)2}
\P_{\delta_x}\langle u_{\infty}, X(1)\rangle=
\Pi_{x}(u_{\infty}(\xi_1)e_{\beta}(1))>0.
\end{equation}
\eqref{P-<h X_t>=02} implies that
$\P_{\delta_x}(\lim_{t\to\infty}\langle h, X_t\rangle=0)<1$. Thus we have
$u_{\infty}(x)>0$.

Now we prove that if $u_{\infty}=0$ almost everywhere, then
$u_{\infty}\equiv 0$. By \eqref{E-u-X(1)2}, we know that $\P_{
\delta_x}\langle u_{\infty}, X(1)\rangle=0$, and thus $\langle
u_{\infty}, X(1)\rangle=0$, $\P_{\delta_x}$-a.s. By \eqref{P-<h
X_t>=02},
$$
\P_{\delta_x}(\lim_{t\to\infty}\langle
h, X_t\rangle=0)=1.
$$
Hence $u_{\infty}(x)=0$ for every $x\in\R^d$.
\qed
\subsection{Proof of Theorem \ref{main2}}
Since $\beta\in {\bf K}_\infty(\xi)$,
by the Gauge Theorem (see \cite[Theorem 2.2]{Song:2002} or \cite[Theorem 2.6]{Chen:2002}),
the assumption that $g_{\beta}$ is not identically infinite
implies that $g_{\beta}$ is bounded between two
positive numbers. By  \cite[Corollary 2.16]{Chen:2002}, we have
$$
\Pi_x\left[\sup_{0\le t\le\infty}e_{\beta}(t)\right]<\infty,\quad\forall x\in\R^d.
$$
By dominated convergence,
$$
g_{\beta}(x)=\lim_{R\to\infty}
\Pi_{x}(e_{\beta}(\tau_{B(0,R)})),\quad x\in\R^d.
$$
Take $h=g_{\beta}$. We know that $h$ is a bounded solution of
$(L+\beta) u=0$ and satisfies \eqref{bounded-h};
by Lemma \ref{u-positive} we only need to prove that if for every $x\in\R^d$,
$
\P_{\delta_{x}}\left(\lim_{t\to\infty}\|X_t\|=0\right)<1,$ then
\begin{equation}\label{*<infty2}\Pi_x\int^{\infty}_0e_{\beta}(s)k(\xi_s)\, \mathrm{d}s<
\infty,\quad x\in\R^d.\end{equation}

First note that the assumption
that $\P_{\delta_{x}}(\lim_{t\to\infty}\|X_t\|=0)<1,
x\in\R^d$ implies that $u_{ch}(x)=-\log
\P_{\delta_x}\exp(-c\lim_{t\to\infty}\langle h, X_t\rangle)>0$ for
every $x\in \R^d$.

Since  $u_{ch}(s, x)\ge u_{ch}(x)$ for every $s\in[0,t]$ and $x\in \R^d$,
by \eqref{u_ch-equi3}, we have
$$\Pi_x\int^t_0e_{\beta}(s)k(\xi_s)u^2_{ch}(\xi_s)\, \mathrm{d}s\le ch(x),\quad x\in\R^d.$$
Letting $t\to\infty$, we get
$$\Pi_x\int^{\infty}_0e_{\beta}(s)k(\xi_s)
u^2_{ch}(\xi_s)\, \mathrm{d}s\le ch(x),\quad x\in\R^d,$$ which can be rewritten
as
\begin{equation}
\label{D-int-u_c2}\int_{\R^d}G_{\beta}(x,y)k(y) u^2_{ch}(y)m(\mathrm{d}y)\le
ch(x),\quad x\in\R^d.
\end{equation}
Letting $R\to\infty$ in
\eqref{D-u_ch-above2}, one gets
$$
u_{ch}(x)\le h(x)\liminf_{R\to\infty}u_{ch}( R).
$$ Since $u_{ch}(x)>0$ and
$0<h(x)<\infty$, we have $\liminf_{R\to\infty}u_{ch}(R)>0$. Then
\eqref{D-int-u_c2} implies \eqref{*<infty2}. \qed

\section{Appendix A: Construction and path regularity}

{\bf Proof of Theorem \ref{existence}} \quad
Let $D_n, n\ge 1$, be a sequence of smooth bounded domains such that $D_n\uparrow \R^d.$
According to Dynkin [7], for each $n$,
the $(L|_{D_n}-\beta^{-},\beta^{+}\wedge n, k)$-superdiffusion
$(X^n_t, t\ge 0)$ exists, where $L|_{D_n}$ is the generator of the process $\xi$
 killed upon leaving $D_n$, and  $\beta^+$ and $\beta^-$ are the positive and
negative parts of $\beta$, respectively. Also note that $(X^n_t, t\ge 0)$ can be
regarded as an $(L|_{D_n},\beta\wedge n, k)$-superdiffusion.

Let $f$ be a positive bounded measurable function on $\R^d$. According to
Dynkin \cite{Dynkin:1993}, for each $n$, there exists a unique bounded solution $u_n$
to the following integral equation:
$$u_n(t,x)+\Pi_x\int^{t\wedge\tau_{n}}_0[-(\beta(\xi_s)\wedge n)
u_n(t-s, \xi_s)+k(\xi_s)u^2( t-s, \xi_s)]\mathrm{d}s=\Pi_x[f(\xi_t), t<\tau_{n}],$$
where $\tau_n$ is the first exit time of the diffusion $\xi$ from $D_n$. We rewrite
the above equation in the following form (according to a result similar to our Lemma 3.1):
\begin{equation}\label{int0}
u_n(t,x)+\Pi_x\int^{t\wedge\tau_{n}}_0e_{\beta^+\wedge n}(s)
[\beta^-(\xi_s)u_n(\xi_s, t-s)+k(\xi_s)u^2(\xi_s, t-s)]\mathrm{d}s=
\Pi_x[e_{\beta^+\wedge n}(t)f(\xi_t), t<\tau_{n}].
\end{equation}
 By the (weak) parabolic maximum
principle (see \cite[p. 128]{Lie} for example),  $u_n$ is increasing.
Let $u_n(t, x)\uparrow u(t, x)$ as $n\uparrow\infty$. Letting $n\to\infty$ in
the above integral equation,
we get
\begin{equation}\label{int}u(t,x)+\Pi_x\int^{t}_0e_{\beta^+}(s)[\beta^-(\xi_s)
u( t-s, \xi_s)+k(\xi_s)u^2( t-s, \xi_s)]\mathrm{d}s=\Pi_x[e_{\beta^+}(t)f(\xi_t)]
\end{equation}
The assumption that $\beta$ is in the Kato class implies that
$u(t, x)\le \Pi_x[e_{\beta^+}(t)f(\xi_t)]\le e^{c_1+c_2t}$ for some positive constants.

To see the minimality of $u$, let $v$ be an arbitrary nonnegative measurable solution
to (\ref{int}). By the (weak) parabolic maximum
principle, $v|_{D_{n}}\ge u_n$ for all $n\ge 1$, and thus $v\ge u$ on $\R^d$.

Equation \eqref{int} can be rewritten as
\begin{equation}\label{int2}u(t,x)+\Pi_x\int^{t}_0[-\beta(\xi_s)u(t-s,\xi_s )+k(\xi_s)u^2( t-s, \xi_s)]\mathrm{d}s=\Pi_x[f(\xi_t)].\end{equation}
Then following the arguments in Appendix A of Engl\"ander and Pinsky [9],
we can get the existence of our superdiffusion.$\hfill\square$

\begin{remark}\label{uniqueness}\rm If $k\in{\bf K}(\xi)$ as well, then using  Gronwall's lemma,
$u$ is the {\it unique}  solution (bounded on any finite interval) of the  integral equation \eqref{int2}.
\end{remark}

Before turning to the proof Theorem \ref{cadlag}, we remark  that
\cite[Appendix A]{Li:2011} explains some important concepts
(e.g. Ray cone, Ray topology) we will be working with, and that \cite[Chap. 5]{Li:2011}
discusses regularity properties of superdiffusions, using similar methods, albeit under different
assumptions on the nonlinear operator.

For the proof we first need a lemma.
The function $f$ is called\footnote{In \cite{Li:2011} a slightly different terminology is followed.}
{\it $\alpha$-supermedian} relative to $P^0_t$ for  $\alpha>0$, if  $e^{-\alpha t}P^0_tf\le f$
for $t\ge 0$.
\begin{lemma}
Assume that $\beta\in {\bf K}(\xi)$ satisfies $\beta\le B$  for some constant $B>0$, and $f$ is
$\alpha$-supermedian relative to $P^0_t$ for some $\alpha>0$. Then
for every $\mu\in M(\mathbb R^d)$,

(i)
$M_t:=e^{-(B+\alpha) t}\langle f,X_t\rangle$ is a $\mathbb P_{\mu}$-supermartingale.

(ii) $\mathbb P_\mu\left(\sup_{0\le r\le t, r\in \mathbb{Q}}\langle 1, X_t\rangle<\infty
\mbox{ for all }t>0\right)=1$.
\end{lemma}

\pf
(i) It is easy to see that it suffices to check
\begin{equation}\label{need.this}
\mathbb P_{\nu}(M_t)\le M_0=\langle f,\nu\rangle,\ t>0,\ \forall\nu\in M(\mathbb R^d).
\end{equation}
This is because for $0\le s<t$, by the Markov property at time $s$,
$$
\mathbb P_{\mu}\left( e^{-B t}\langle f,X_t\rangle\mid \mathcal{F}_s\right)=
\mathbb P_{X_{s}} M_{t-s}e^{-(B+\alpha) s}\le\langle f,X_s
\rangle e^{-(B+\alpha)s}=M_s,
$$
where in the last inequality above we used (\ref{need.this}) with $\nu=X_s$.
Using the assumption that $f$ is $\alpha$-supermedian, we obtain
 $$\mathbb P_{\delta_{x}}M_t=e^{-(B+\alpha) t}(P^\beta_t f)(x)\le
e^{-\alpha t}P^0_tf(x)\le f(x).$$
Therefore (\ref{need.this}) holds.

(ii) By the proof of Theorem \ref {th1},  there are $a,\gamma>0$ and
a sufficiently small $T>0$
such that $M_r:=e^{\gamma t}\langle 1, X_r\rangle$ satisfies
$$
\mathbb P_{\mu}[M_r|{\cal F}_s]\ge a M_s, \quad 0\le s\le r\le T
\mbox{ with } r,s \in\mathbb{Q}.
$$ Then by Doob's inequality (Lemma \ref{modified.Doob} in discrete time),
$$
\mathbb P_\mu\left(\sup_{0\le r\le T,\, r\in \mathbb{Q}}\langle 1, X_r\rangle>K\right)
\le (aK)^{-1}\P_\mu M_t\le (aK)^{-1}e^{(\gamma+B)T}.
$$
Letting $K\uparrow 0$, we see that for any fixed $t>0$, $\mathbb P_\mu(\sup_{0\le r
\le T, \,r\in \mathbb{Q}}\langle 1, X_r\rangle=\infty)=0$. Since we can split $[0,\infty)$
to intervals of length $T$,  the result of (ii) holds.
\qed

{\bf Proof of Theorem \ref{cadlag}} \quad Let $(\overline{\mathbb R}^d,
\overline{{\cal B}(\mathbb R^d)})$ be the Ray-Knight compactification of
$(\mathbb R^d, {\cal B}(\mathbb R^d))$ associated with the semigroup
$\{P^0_t: t\ge 0\}$
and a suitably chosen countable Ray cone (see the last paragraph on \cite[p. 342]{Fitz}),
and let $M_r(\overline{\mathbb R}^d)$ be the space of finite measures on
$\overline{\mathbb R}^d$ with the weak Ray topology. Suppose $W$ is the
space of right continuous paths from $[0,\infty)$ to $M_r(\mathbb R^d)$
with left limits in $M_r(\overline {\mathbb R}^d)$, where $M_r(\mathbb R^d)$
carries the relative topology inherited from $M_r(\overline{\mathbb R}^d)$. We
write $\tilde X=(\tilde X_t, t\ge 0)$ for the coordinate process on $W$ and put
${\cal G}=\sigma\{\tilde X_t; t\ge 0)$.  Using the above lemma, the argument
in the proof of
\cite[Theorem 2.11]{Fitz} is applicable to our setup, so for any given
$\mu\in M(\mathbb R^d)$ there exists a unique probability measure
${\bf P}_\mu$ on $(W, {\cal G})$ such that ${\bf P}_\mu(\tilde X_0=\mu)=1$ and
$(\tilde X_t, t\ge 0)$ under ${\bf  P}_\mu$ has the same law as the
superprocess $X$ under $\mathbb P_\mu$.

As before, let $M(\mathbb R^d)$ denote the space of finite measures on $\mathbb R^d$ with the weak topology, induced by the mappings $\langle f, \tilde X_t\rangle$
as $f$ runs through the bounded continuous functions on $\mathbb R^d$.
(The Borel $\sigma$-algebras on $M_r(\mathbb R^d)$ and on $M(\mathbb R^d)$
both coincide with ${\cal M}$.) Since the diffusion process $\xi$ is continuous,
using the arguments of \cite[Section 3]{Fitz}, we have that if $f$ is a bounded
continuous function on $\mathbb R^d$, then $\langle f, \tilde X_{\cdot}\rangle$
is right continuous on $[0,\infty)$ almost surely; and if $f(\xi_{\cdot})$ has left
limits on $[0,\infty)$ almost surely, then so does $\langle f, \tilde X_{\cdot}\rangle$.
That is to say, $\tilde X$ is a $c\grave{a}dl\grave{a}g$ process on
the state space $M(\mathbb R^d)$.
\qed

\section{Appendix B: Review on Feynman-Kac semigroups and Gauge Theory}\label{sec-FK}

Recall that $\beta$ is in the Kato class ${\bf K}(\xi)$. In this appendix we present
some background material on the Feynman-Kac semigroup. Recall from Section 1 that
$$
P^\beta_tf(x):=\Pi_x[e_\beta(t)f(\xi_t)],
$$
and that $\{P^\beta_t, t\ge 0\}$ is a strongly
continuous semigroup on $L^p(\R^d, m)$ for  $1\le p<\infty$.

For any domain $D\subset\R^d$ and $x\in D$, we will use
$\delta_D(x)$ to denote the distance from $x$ to $D^c$:
$\delta_D(x):=\inf\{|x-y|: y\in D^c\}$.
Let $\xi^D$ be
the subprocess of $\xi$ killed upon exiting $D$.
It is well known that $\xi^D$ has a transition density $p_D(t, x, y)$
with respect to the Lebesgue measure.
We will use
$\{P^{\beta, D}_t, t\ge 0\}$ to denote the semigroup of $\xi^{D}$:
$$
P^{\beta, D}_tf(x):=\Pi_x[e_\beta(t)f(\xi_t), t<\tau_D],
$$
where
$$
\tau_D=\inf\{t>0: \xi_t\notin D\}.
$$

When $D^c$ is non-polar, that is, when $\Pi_x(\tau_D<\infty)$ is not
identically zero, $\xi^D$ is transient.
In this case, the function $G_D(x, y):=\int^\infty_0p_D(t, x, y)\, \mathrm{d}t$
is  well defined and is called  the Green's function of $\xi^{D}$ with respect to
the Lebesgue measure.
Then $\widetilde{G}_D(x, y):=G_{D}(x, y)/m(y)$
is the Green's function of $\xi^{D}$ with respect to $m(y)\mathrm{d}y$.

For any $n\ge 1$, put $D_n=B(0, n)$. We will use the shorthand $\xi^{(n)}$ to denote
$\xi^{D_n}$ and $G_{n}$ to denote $G_{D_n}$.
It follows from \cite{Hub-Sie, kim-son2} that $G_n$  is comparable to the
Green's function of the killed Brownian motion in $D_n$.
Therefore we have the following result.
\begin{proposition}\label{p:new1} There exists $c_1=c_1(n, d)>1$ such that when $d\ge 3$,
\begin{equation}\label{e:ge32}
c^{-1}_1\left(1\wedge\frac{\delta_B(x)\delta_B(y)}{|x-y|^2}\right)\le
G_{B}(x, y)\le c_1\frac1{|x-y|^{d-2}}\left(1\wedge
\frac{\delta_B(x)\delta_B(y)}{|x-y|^2}\right), \quad x, y\in B
\end{equation}
for any ball $B\subset D_n$; when $d=2$
\begin{equation}\label{e:ge22}
c^{-1}_1\log\left(1+\frac{\delta_B(x)\delta_B(y)}{|x-y|^2}\right)\le
G_{B}(x, y)\le
c_1\log\left(1+\frac{\delta_B(x)\delta_B(y)}{|x-y|^2}\right), \quad
x, y\in B
\end{equation}
for any ball $B\subset D_n$; and when $d=1$
\begin{equation}\label{e:ge12}
c^{-1}_1(\delta_B(x)\wedge \delta_B(y))\le G_B(x, y)\le
c_1(\delta_B(x)\wedge \delta_B(y)), \quad x, y\in B
\end{equation}
for any ball $B\subset D_n$.
\end{proposition}
\subsection{The 3G inequalities and the Martin kernel}

Recall that $u$ is defined by \eqref{e:gfn4bm}.
Using \eqref{e:ge32}--\eqref{e:ge12}, we can easily get the following.

\begin{proposition}[The 3G inequalities]\label{p:new2}
There exists $c=c(d,n)$ such that, when $d\ge 3$,
\begin{equation}\label{e:3g32}
\frac{G_{B}(x, y)G_{B}(y, z)}{G_{B}(x, z)}\le c(u(x-y)+u(y-z)),
\quad x, y, z\in B
\end{equation}
for any ball $B\subset D_n$; when $d=2$,
\begin{equation}\label{e:3g22}
\frac{G_B(x, y)G_B(y, z)}{G_B(x, z)}\le c[(1\vee u(x-y))+(1 \vee
u(y-z))], \quad x, y, z\in B
\end{equation}
for any ball $B\subset D_n$; and when $d=1$,
\begin{equation}\label{e:3g12}
\frac{G_B(x, y)G_B(y, z)}{G_B(x, z)}\le c, \quad x, y, z\in B
\end{equation}
for any ball $B\subset D_n$.
\end{proposition}

\pf The $d\ge 3$ case follows from \cite[Theorem 6.5]{Chung:1995}, the
$d=2$ case follows from \cite[Theorem 6.15]{Chung:1995}, while $d=1$ follows
from direct calculation. \qed

The three inequalities in Proposition \ref{p:new2}
are called {\it 3G
inequalities}.
For any ball $B$ and $x_0\in B$, the Martin kernel $M_B(x, z), (x, z)\in B\times
\partial B,$ based at $x_0$ is defined by
$$
M_B(x, z):=\lim_{B\ni y\to z\in\partial B}\frac{G_B(x, y)}{G_B(x_0,
y)}.
$$
The base $x_0$ plays no essential role here.
One then can easily deduce the following result from the 3G inequalities above.
\begin{proposition}\label{p:new3}
There exists $c=c(d, n)>0$ such that, when
$d\ge 3$,
\begin{equation}\label{e:3g3k}
\frac{G_{B}(x, y)M_{B}(y, z)}{M_{B}(x, z)}\le c(u(x-y)+u(y-z)), \quad x, y\in B, z\in \partial B
\end{equation}
for every ball $B\subset D_n$; when $d=2$,
\begin{equation}\label{e:3g2k}
\frac{G_{B}(x, y)M_{B}(y, z)}{M_{B}(x, z)}\le c[(1\vee u(x-y))+(1 \vee u(y-z))],
\quad x, y\in B, z\in \partial B
\end{equation}
for every ball $B\subset D_n$; when $d=1$,
\begin{equation}\label{e:3g1k}
\frac{G_{B}(x, y)M_{B}(y, z)}{M_{B}(x, z)}\le c, \quad x, y\in B, z\in \partial B
\end{equation}
for every ball $B\subset D_n$.
\end{proposition}

The following result is proved in \cite{kim-son1, kim-son2}.

\begin{proposition}\label{p:new4} For any $n\ge 1$, there exist $c_i=c_i(n)>1$, $i=1, 2$,
such that the transition density $p^{(n)}_t$ of $\xi^{(n)}$ with respect to the Lebesgue measure
satisfies
\begin{eqnarray}\label{e:dhke}
&&c^{-1}_1t^{-d/2}\left(1\wedge\frac{\delta_n(x)}{\sqrt{t}}\right)\left(1\wedge
\frac{\delta_n(x)}{\sqrt{t}}\right)
e^{-\frac{c_2|x-y|^2}{t}}\le p^{(n)}_t(x, y)\nonumber\\
&&\le c_1t^{-d/2}\left(1\wedge\frac{\delta_n(x)}{\sqrt{t}}\right)\left(1\wedge
\frac{\delta_n(x)}{\sqrt{t}}\right)
e^{-\frac{|x-y|^2}{c_2t}}
\end{eqnarray}
for all $(t, x, y)\in (0, 1]\times D_n\times D_n$.
\end{proposition}

We then have the following result.
\begin{proposition}\label{p:kato2kato}
If $\beta\in {\bf K}(\xi)$, then for any $n\ge 1$,
$$
\lim_{r\to 0}\sup_{x\in D_n}\int_{|y-x|<r}u(x-y)|\beta(y)|\mathrm{d}y=0.
$$
\end{proposition}

\pf It follows from \eqref{e:dhke} that there exist constants $c_1,
c_2>1$ such that for any $(t, x, y)\in (0, 1]\times D_{n}\times
D_{n}$,
$$
p^{(n+1)}_t(x, y)\ge c_1^{-1}\exp\left\{-\frac{c_2|x-y|^2}{t}\right\}.
$$
Since
$$
\int^t_0\Pi_x[|\beta(\xi_s)|]\mathrm{d}s\ge \int^t_0\int_{D_{n}}
p^{(n+1)}_s(x, y) |\beta(y)|\mathrm{d}y\mathrm{d}s,
$$
we can apply the arguments in the proof of \cite[Lemma 3.5]{Chung:1995} and the first
part of the proof of  \cite[Theorem 3.6]{Chung:1995} to get the conclusion of our proposition.
\qed
\subsection{Probabilistic representation of $\lambda_2$}
The following result is a generalization of \cite[Theorem
4.4.4]{Pinsky:1995} and it implies that \eqref{prob-lambda} is valid
when $\beta\in {\bf K}(\xi)$.

\begin{proposition}[Probabilistic representation of $\lambda_2$]\label{p:prob-lambda}
Let $\{D_n\}_{n\ge 1}$ be an increasing sequence of bounded domains with
$D_n\uparrow \mathbb R^d$ as $n\to\infty.$ If
$\tau_n:=\inf_{t\ge 0}\{t:\xi_t\not\in D_n\},\ n\ge 1,$ then
$$
\lambda_2(\beta)=\sup_n\lim_{t\to\infty}\frac1t\log\sup_{x\in D_n}
\Pi_x(e_\beta(t); t< \tau_n).
$$
\end{proposition}

\pf
Let $P^{\beta, n}_t$ stand for $P^{\beta, D_n}_t$ and let
$$
\lambda_2^n:=\lim_{t\to\infty}\frac1{t}\log \|P^{\beta, n}_t\|_{2, 2},
$$
where $\|P^{\beta, n}_t\|_{2, 2}$ stands for the operator norm of $P^{\beta, n}_t$ from $L^2(D_n, m)$
to $L^2(D_n, m)$.
It is well known (see, for instance, \cite{Chen:2011}) that
\begin{equation}\label{e:l2byf1}
-\lambda_2(\beta)=\inf\left\{\frac12\int_{\R^d}(\nabla f a\nabla f)
e^{2Q}\mathrm{d}x-\int_{\R^d}f^2\beta e^{2Q}\mathrm{d}x:
f\in C^\infty_c(\R^d), \|f\|_2=1\right\}
\end{equation}
and
\begin{equation}\label{e:l2byf2}
-\lambda_2^n(\beta)=\inf\left\{\frac12\int_{\R^d}
(\nabla f a\nabla f)e^{2Q}\mathrm{d}x-\int_{\R^d}f^2\beta e^{2Q}\mathrm{d}x:
f\in C^\infty_c(D_n), \|f\|_2=1\right\}.
\end{equation}

For any $n\ge 1$, by using \eqref{e:ge32}--\eqref{e:ge12} and
Proposition \ref{p:kato2kato} we can easily see that $\beta\in{\bf K}_\infty(\xi^{(n)})$
(The definition of the Kato class ${\bf K}_\infty(\xi^{(n)})$ is similar
to Definition \ref{Kinfty}; see \cite{Song:2002} for details.). Thus
it follows from \cite[Theorem 2.3]{Chen:2011} that for any $n\ge 1$,
$$
-\lambda_2^n(\beta)=\lim_{t\to\infty}\frac{1}{t}\log\sup_{x\in D_n}P^{\beta, n}_t1(x).
$$
Since $\lambda_2^n(\beta)\to \lambda_2(\beta)$, combining the above with
\eqref{e:l2byf1}--\eqref{e:l2byf2} yields
the conclusion of our proposition. \qed
\subsection{Properties of the gauge function}

Recall that the gauge function $g_\beta$  is defined in Definition \ref{def.of.Gauge.function}.
For any open set
$D\subset\R^d$ and nonnegative measurable function $f$ on $\partial D$, we define
$$
g^D_{\beta,
f}(x):=\Pi_x[e_{\beta}(\tau_D)f(\xi_{\tau_D})1_{\{\tau_D<\infty\}}], \quad x\in D.
$$
The Harnack-type inequalities in the following result will be used later.

\begin{lemma}\label{Gauge-positive} (1) For any open set
$D\subset\R^d$ and nonnegative measurable function $f$ on $\partial D$, if the function
$g^D_{\beta, f}$
is not identically infinite on $D$, then for any compact set $K$,
$g^D_{\beta, f}$ is bounded on $K$ and  there
exists $A=A(D, K, \beta)>1$, independent of $f$, such that
\begin{equation}\label{e:HI}
\sup_{x\in K}g^D_{\beta, f}(x)\le A \inf_{x\in K}g^D_{\beta, f}(x).
\end{equation}
Furthermore, $g^D_{\beta, f}$ is a continuous solution of
$(L+\beta) h=0$ in $D$ in the sense of distributions.

(2) If $g_\beta$ is not identically infinite in $\R^d$, then for
any compact set $K\subset \mathbb R^d$,  $g_\beta(x)$ is bounded on $K$
and there exists an $A=A(K, \beta)>1$ such that
\begin{equation}\label{e:HI2}
\sup_{x\in K}g_\beta(x)\le A \inf_{x\in K}g_\beta(x).
\end{equation}
Furthermore, $g_{\beta}$ is a continuous solution of
$(L+\beta) h=0$ in $\R^d$ in the sense of distributions.

(3) If $g_\beta$ is not identically zero in $\R^d$,
then $g_\beta(x) >0$  for all $x\in \R^d$.
\end{lemma}

\pf (1) The proof follows the same line
of arguments as that of \cite[Theorem 5.18]{Chung:1995}.
Without loss of generality, we may and do assume
that $K\subset B(0, n)$ and that there exists $x_1\in K$ such
that $g^D_{\beta, f}(x_1)<\infty$. Then, by the definition of $g^D_{\beta, f}$
and the strong Markov property, for any ball $B=B(x_1, r)\subset \overline{B(x_1, r)}\subset D$, we have
$$
g^D_{\beta, f}(x_1)=\Pi_{x_1}[e_\beta(\tau_B)g^D_{\beta,
f}(\xi_{\tau_B})].
$$
By \eqref{e:3g3k}--\eqref{e:3g1k} and Proposition \ref{p:kato2kato}, for any $\epsilon>0$,
we can choose $r_0=r_0(n, \beta)\in (0, 1]$ such that for any $r\in (0, r_0)$ and
any $(x, z)\in B\times \partial B$:
$$
\Pi^z_x\int^{\tau_B}_0e_{|\beta|}(t)\, \mathrm{d}t\le \frac12,
$$
where $\Pi^z_x$ stands for the law of the $M_B(\cdot,
z)$-conditioned diffusion, i.e., the process such that for all
bounded Borel function on $B$ and $t>0$,
$$
\Pi^z_x[f(\xi_t)]= \frac1{M_B(x, z)}\Pi_x[f(\xi_t)M_B(\xi_t, z);
t<\tau_B].
$$
Repeating the argument of \cite[Theorem 5.17]{Chung:1995}, we get
that
$$
\frac12\le \Pi^z_xe_\beta(\tau_B)\le 2.
$$
Put $v(x, z):=\Pi^z_xe_\beta(\tau_B)$, then  by \cite[Proposition
5.12]{Chung:1995} (which is also valid for $\xi$ by the same
arguments contained in \cite[Section 5.2]{Chung:1995}) we have
$$
g^D_{\beta, f}(x_1)=\int_{\partial B}v(x_1, z)K_B(x_1, z)g^D_{\beta,
f}(z)\,\sigma(\mathrm{d}z)
$$
where $\sigma$ stands for the surface measure on $\partial B$ and
$K_B$ is the Poisson kernel of $B$ with respect to $\xi$. It follows
from the Harnack inequality (applied to the harmonic functions of $\xi$)  that
there exists  some $c>1$ such that
$$
\sup_{x\in B(x_1, r/2)}K_B(x, z))\le c \inf_{x\in B(x_1, r/2)}K_B(x,
z),\qquad \forall z\in \partial B.
$$
Since, for $x\in B$,
$$
g^D_{\beta, f}(x)=\int_{\partial B}v(x, z)K_B(x, z)g^D_{\beta,
f}(z)\sigma(\mathrm{d}z),
$$
therefore we have
\begin{equation}\label{e: newHarnack}
\sup_{x\in B(x_1, r/2)}g^D_{\beta, f}(x)\le c \inf_{x\in B(x_1,
r/2)}g^D_{\beta, f}(x).
\end{equation}
Now \eqref{e:HI} follows from a standard chain argument.
In fact, for any compact subset $K$ of $D$, there exist $r\in (0, 1]$ and
an integer $N>1$ such that, for any $x, x'\in K$, there exists a subset
$\{y_i: i=1, \dots, l\}$, $1\le l\le N$,  with $\overline{B(y_i, r)}\subset D$, $i=1, \dots, l$, and
$$
|x-y_1|<\frac{r}2, \quad |y_i-y_{i+1}|<\frac{r}2, \quad i=1, \dots, l-1, \quad |x'-y_l|<\frac{r}2.
$$
Applying \eqref{e: newHarnack} repeatedly, we arrive at \eqref{e:HI}.
The last
assertion of (1) can be proved by repeating the argument of the
Corollary to \cite[Theorem 5.18]{Chung:1995} and we omit the
details.

(2) The proof of (2) is similar to that of (1).

(3) The proof of this part is similar to
that of \cite[Proposition
8.10]{Chung:1995} and we omit the details. \qed
\subsection{The operator $G^{\beta}$}
 For any  $f\ge 0$ on
$\R^d$, set
\begin{equation}\label{G-qf}
G^\beta f(x):=\Pi_{x}\int^{\infty}_0e_{\beta}(s)f(\xi_s)\, \mathrm{d}s.
\end{equation}
$G^0f$ will be denoted as $Gf$. The following result will be needed later.

\begin{lemma}\label{inequ-green} Suppose that $f\ge 0$  is locally bounded on $\R^d$.
If there exists an $x_1\in\R^d$ such that $G^\beta f(x_1)<\infty$, then $G^\beta f$
is locally bounded on $\R^d$.
\end{lemma}

\pf The proof is similar to that of the first part of Lemma
\ref{Gauge-positive}.
For convenience, we put $\widetilde f:=G^\beta f$ in this proof.
Without the loss of generality, we may and do assume
that the compact set $K$ satisfies $K\subset B(0, n)$, and furthermore, that
there exists an $x_1\in K$ such that $\widetilde f(x_1)<\infty$. Let $v(x,
z):=\Pi^z_xe_\beta(\tau_B)$. By the strong Markov property, for
any $B=B(x_1, r)$, we have
\begin{equation}\begin{array}{rl}
\widetilde
f(x_1)&=\Pi_{x_1}\displaystyle\int^{\tau_B}_0e_{\beta}(s)f(\xi_s)\mathrm{d}s+
\Pi_{x_1}\left[e_\beta (\tau_B)\Pi_{\xi_{\tau_B}}
\displaystyle\int^{\infty}_0e_{\beta}(s)f(\xi_s)\mathrm{d}s\right]\\
&=\Pi_{x_1}\displaystyle\int^{\tau_B}_0e_{\beta}(s)f(\xi_s)\mathrm{d}s
+\displaystyle\int_{\partial B}v(x_1,z)K_B(x_1, z)\widetilde
f(z)\sigma(\mathrm{d}z).\end{array}
\end{equation}
By \eqref{e:3g3k}--\eqref{e:3g1k}, Proposition \ref{p:kato2kato} and
the argument of \cite[Theorem 5.17]{Chung:1995}, for any
$\epsilon>0$, we can choose $r_0=r_0(n, \beta)\in (0, 1]$ such that
for any $r\in (0, r_0)$ and any $(x, z)\in B\times \partial B$:
$$
\frac{1}{2}\le \Pi^z_x[e_{\beta}(\tau_B)]\le
\Pi^z_x[e_{|\beta|}(\tau_B)]\le 2;\quad \Pi_{x}\tau_B^2\le 2;\quad
\Pi_{x}[e_{2|\beta|}(\tau_B)]\le 2.
$$
We then have
$$
\widetilde  f(x_1)\ge \frac{1}{2}\int_{\partial B}K_B(x_1,z)\widetilde
f(z)\sigma(\mathrm{d}z)
$$
and
$$\begin{array}{rl}
\widetilde  f(x)&=\Pi_{x}\displaystyle\int^{\tau_B}_0e_{\beta}(s)f(\xi_s)\mathrm{d}s
+\displaystyle\int_{\partial B}v(x,z)K_B(x,
z)\widetilde  f(z)\sigma(\mathrm{d}z)\\
&\le C\Pi_{x}(\tau_B e_{|\beta|}(\tau_B))+\displaystyle\int_{\partial
B}v(x,z)K_B(x, z)\widetilde  f(z)\sigma(\mathrm{d}z)\\
&\le
C[\Pi_{x}\tau_B^2]^{1/2}[\Pi_{x}[e_{2|\beta|}(\tau_B)]]^{1/2}+\displaystyle\int_{\partial
B}v(x,z)K_B(x, z)\widetilde  f(z)\sigma(\mathrm{d}z),\end{array}
$$
where $C$ is the upper bound of $f$ on $B$. It follows from the
Harnack inequality (for harmonic functions of $\xi$) that there exists some
$c>1$ such that
$$
\sup_{x\in B(x_1, r/2)}K_B(x, z))\le c \inf_{x\in B(x_1, r/2)}K_B(x,
z).
$$
Thus
$$
\sup_{x\in B(x_1, r/2)}\widetilde  f(x)\le 2C +4c \widetilde  f(x_1).
$$
Now the assertion of the lemma follows from
a standard chain argument, as was done in the proof of Lemma \ref{Gauge-positive}(1).
\qed

{\bf Acknowledgement}
The first author owes thanks to Zenghu Li for valuable discussions about path regularity
questions, and to Peking University for their  hospitality when visiting Y. Ren.
We also thank the two referees for several helpful comments and suggestions on the first version
of this paper.

\vskip 0.3truein

 \noindent {\bf J\'anos Engl\"ander:} Department of Mathematics,
University of Colorado,
Boulder, CO 80309-0395, U.S.A.
Email: {\texttt janos.englander@colorado.edu},

\noindent{\tt http://euclid.colorado.edu/~englandj/MyBoulderPage.html}

 \bigskip
\noindent{\bf Yan-Xia Ren:} LMAM School of Mathematical Sciences \& Center for
Statistical Science, Peking
University,  Beijing, 100871, P.R. China. Email: {\texttt
yxren@math.pku.edu.cn},

\noindent{\tt http://www.math.pku.edu.cn/teachers/renyx/indexE.htm}

\bigskip
\noindent {\bf Renming Song:} Department of Mathematics,
University of Illinois,
Urbana-Champaign, IL 61801, U.S.A.
Email: {\texttt rsong@math.uiuc.edu}, {\tt http://www.math.uiuc.edu/~rsong/}

\end{document}